\documentclass[11pt]{article}

\usepackage{amsfonts,latexsym,amsthm,amssymb,amsmath,amscd,euscript,tikz,mathtools}
\usepackage{framed}
\usepackage{fullpage}
\usepackage{color}
\usepackage[colorlinks=true,citecolor=blue,linkcolor=blue]{hyperref}
\usepackage{enumitem}
\usepackage{siunitx}
\usepackage{textcomp}
\usepackage{physics}
\usepackage{mathrsfs}
\usepackage{dsfont}
\usepackage[ruled,vlined]{algorithm2e}
\usepackage{titlesec}
\usepackage{tikz-cd}
\usepackage{thmtools}
\usepackage{thm-restate}
\usepackage{cleveref}
\usepackage{graphicx}

\allowdisplaybreaks[1]

\newtheorem{theorem}{Theorem}
\newtheorem{proposition}[theorem]{Proposition}
\newtheorem{lemma}[theorem]{Lemma}

\newtheorem{corollary}[theorem]{Corollary}

\theoremstyle{definition}
\newtheorem{definition}[theorem]{Definition}

\theoremstyle{remark}
\newtheorem*{remark}{Remark}

\newcommand{\cG}{\mathcal{G}}

\newcommand{\cN}{\mathcal{N}}

\newcommand{\cW}{\mathcal{W}}

\newcommand{\bC}{\mathbb{C}}
\newcommand{\bE}{\mathbb{E}}\newcommand{\bF}{\mathbb{F}}

\newcommand{\bN}{\mathbb{N}}

\newcommand{\bR}{\mathbb{R}}

\newcommand{\bZ}{\mathbb{Z}}

\newcommand{\1}{\mathds{1}}

\newcommand{\val}{\operatorname{val}}
\newcommand{\RW}{\operatorname{RW}}
\newcommand{\Bin}{\operatorname{Bin}}

\newcommand{\poly}{\operatorname{poly}}

\newcommand{\disTV}{d_{\mathrm{TV}}}

\newcommand{\disltwo}{d_{\ell_2}}

\newcommand{\Var}{\operatorname{Var}\limits}

\newcommand{\bigast}{\mathop{\scalebox{1.5}{\raisebox{-0.2ex}{$\ast$}}}}%

\newcommand{\nc}{\newcommand}

\nc{\on}{\operatorname}
\nc{\Spec}{\on{Spec}}
\nc{\Aut}{\textit{Aut}}
\nc{\id}{\textit{id}}
\nc{\chr}{\on{char}}
\nc{\im}{\on{im}}
\nc{\Hom}{\on{Hom}}
\nc{\lcm}{\on{lcm}}
\nc{\dual}[1]{\prescript{t}{}{#1}}
\nc{\transpose}[1]{{#1}^{\intercal}}
\nc{\Sym}{\on{Sym}}
\nc{\End}{\on{End}}
\nc{\stab}{\on{stab}}
\nc{\Li}{\on{Li}}
\nc{\spn}{\on{span}}
\nc{\sgn}{\on{sgn}}
\nc{\supp}{\on{supp}}
\nc{\Unif}{\on{Unif}}

\title{A New Berry-Esseen Theorem for Expander Walks\footnote{To appear in the 55th Annual ACM Symposium on Theory of Computing (STOC 2023). These results first appeared in the author's undergraduate thesis \cite{golowich_random_2022}.}}
\author{
  Louis Golowich\thanks{Currently supported by an NSF Graduate Fellowship. This work was done while the author was at Harvard University.} \\
  UC Berkeley \\
  Berkeley, CA USA \\
  \href{mailto:lgolowich@berkeley.edu}{\texttt{lgolowich@berkeley.edu}}
}

\begin{document}

\pagenumbering{gobble}

\maketitle

\begin{abstract}
  We prove that the sum of $t$ boolean-valued random variables sampled by a random walk on a regular expander converges in total variation distance to a discrete normal distribution at a rate of $O(\lambda/t^{1/2-o(1)})$, where $\lambda$ is the second largest eigenvalue of the random walk matrix in absolute value. To the best of our knowledge, among known Berry-Esseen bounds for Markov chains, our result is the first to show convergence in total variation distance, and is also the first to incorporate a linear dependence on expansion $\lambda$. In contrast, prior Markov chain Berry-Esseen bounds showed a convergence rate of $O(1/\sqrt{t})$ in weaker metrics such as Kolmogorov distance.
  
  Our result also improves upon prior work in the pseudorandomness literature, which showed that the total variation distance is $O(\lambda)$ when the approximating distribution is taken to be a binomial distribution. We achieve the faster $O(\lambda/t^{1/2-o(1)})$ convergence rate by generalizing the binomial distribution to discrete normals of arbitrary variance. We specifically construct discrete normals using a random walk on an appropriate 2-state Markov chain. Our bound can therefore be viewed as a regularity lemma that reduces the study of arbitrary expanders to a small class of particularly simple expanders.
\end{abstract}

\newpage



\pagenumbering{arabic}

\section{Introduction}

The Berry-Esseen Theorem \cite{berry_accuracy_1941,esseen_liapunov_1942} is the quantitative version of the central limit theorem, and states that the sum of $t$ sufficiently independent random variables converges to a normal distribution at a rate of $O(1/\sqrt{t})$. In this paper, we prove a new Berry-Esseen theorem for random walks on expander graphs. Specifically, we show that the sum of $t$ boolean-valued random variables sampled using a random walk on a regular expander graph converges in total variation distance to an appropriate discrete normal distribution at a rate of $O(\lambda/t^{1/2-o(1)})$, where $\lambda$ is the second largest eigenvalue of the random walk matrix in absolute value. This bound recovers the $O(1/\sqrt{t})$ convergence rate in the classical Berry-Esseen theorem (up to a $t^{o(1)}$ error), and simultaneously incorporates
a linear dependence on $\lambda$.
To the best of our knowledge, prior known Berry-Esseen theorems for Markov chains with spectral gap $1-\lambda$ did not achieve the linear dependence on $\lambda$ in our bound. Furthermore, our bound applies to total variation distance, whereas prior Markov chain Berry-Esseen bounds only considered weaker metrics such as Kolmogorov distance.

\subsection{Main result}
\label{sec:mainresultinf}
This section describes our main result, and interprets it both as a Berry-Esseen theorem and as a regularity lemma for expander walks. The formal result statement is given in Section~\ref{sec:resultstatement}.

In words, our main result says that if $G$ is a $\lambda$-spectral expander graph with vertex labeling $\val:V(G)\rightarrow\{0,1\}$, then the sum $\sum_{i\in[t]}\val(\RW_G^t)_i$ of the labels from a length-$t$ random walk on $G$ converges in total variation distance to an appropriate discrete normal distribution $\cN_{\sigma^2}^t$ at a rate of almost $O(\lambda/\sqrt{t})$.

\begin{theorem}[Main result; informal statement of Theorem~\ref{thm:berryesseen}]
  \label{thm:mainresultinf}
  Let $G$ be a regular $\lambda$-spectral expander graph with vertex labeling $\val:V(G)\rightarrow\{0,1\}$, where $\lambda$ is less than some sufficiently small constant. Then for all $t\in\bN$,
  \begin{equation}
    \label{eq:mainresultinf}
    \disTV\left(\sum_{i\in[t]}\val(\RW_G^t)_i,\; \cN_{\sigma^2}^t\right) \leq \frac{\lambda}{t^{1/2-o_p(1)}}
  \end{equation}
  for an appropriate discrete normal distribution $\cN_{\sigma^2}^t$, where the $o_p(1)$ term above approaches $0$ as $t\rightarrow\infty$ for every fixed value of the weights $p_b=|\val^{-1}(b)|/n$ of the labeling $\val:V(G)\rightarrow\{0,1\}$.
\end{theorem}

The parameter $\sigma^2$ above denotes the {\it asymptotic variance} so that
\begin{equation*}
  \sigma^2 = \lim_{t\rightarrow\infty}\frac{\Var(\Sigma\val(\RW_G^t))}{t} = \lim_{t\rightarrow\infty}\frac{\Var(\cN_{\sigma^2}^t)}{t},
\end{equation*}
where we denote $\Sigma\val(\RW_G^t)=\sum_{i\in[t]}\val(\RW_G^t)_i$.

For intuition, observe that when $\lambda=0$, then $G$ is a complete graph with self loops so that $\Sigma\val(\RW_G^t)=\Bin(t,p_1)$, and the right hand side of~(\ref{eq:mainresultinf}) vanishes. Thus we must have $\cN_{\sigma^2}^t=\Bin(t,p_1)$ with $\sigma^2=p_0p_1$. Therefore our discrete normal $\cN_{\sigma^2}^t$ provides a generalization of the binomial distribution to arbitrary variances for any given $t,p$.

We prove Theorem~\ref{thm:mainresultinf} for a general class of discrete normals satisfying a set of axioms (see Section~\ref{sec:discnormdef}). We then present a concrete instantiation of such a discrete normal family satisfying these axioms, given by letting $\cN_{\sigma^2}^t$ be the distribution $\Sigma\val(\RW_{G_{\nu,p}}^t)$ for an appropriate $\nu=\nu(p,\sigma^2)$, where $G_{\nu,p}$ denotes the \textit{$\nu$-sticky, $p$-biased random walk}. This walk is the 2-state\footnote{To express $G_{\nu,p}$ as a regular graph, we let the graph have arbitrarily large vertex set $V=V_0\sqcup V_1$, but specify that all $|V_b|=p_b|V|$ vertices within each $V_b$ are interchangeable (see Section~\ref{sec:sticky}).} Markov chain with stationary distribution $p$ that moves from state $b\in\{0,1\}$ to state $b'\in\{0,1\}$ with probability $(1-\nu)\cdot p_{b'}+\nu\cdot\1_{b=b'}$.

Our Berry-Esseen result can also be viewed as a regularity lemma for expander walks. The $\nu$-sticky walk has expansion $\lambda(G_{\nu,p})=\nu$, and therefore provides a canonical ``simplest'' example of a $\nu$-spectral expander. Indeed, Guruswami and Kumar~\cite{guruswami_pseudobinomiality_2021} introduced $G_{\nu,p}$ for the case $p_0=p_1=1/2$ as a model for studying expander walks, but with no formal way to derive results about general expanders from $G_{\nu,p}$. Our bound in Theorem~\ref{thm:mainresultinf} with $\cN_{\sigma^2}^t=\Sigma\val(\RW_{G_{\nu(p,\sigma^2)},p}^t)$ shows that the distribution of the sum of the labels from a $t$-step random walk on an arbitrary $\lambda$-spectral expander graph $G$ is approximated by the analogous distribution for $G_{\nu,p}$, up to a total variation error that vanishes as $\lambda\rightarrow 0$ or as $t\rightarrow\infty$. That is, our bound reduces the study of the random walk on an arbitrary expander $G$ to the study of a much simpler sticky walk $G_{\nu,p}$.

Though we do not present any direct applications of our Berry-Esseen bound, similar prior results have found various applications in computer science. Prior limit theorems that approximate sums of random variables in total variation distance have yielded applications to learning theory and algorithmic game theory \cite{daskalakis_efficient_2008,daskalakis_oblivious_2009,daskalakis_learning_2012,daskalakis_learning_2013}. At a high level, these applications make use of the regularity lemma interpretation described above. Specifically, they use appropriate limit theorems to obtain concise descriptions of target distributions. The Berry-Esseen theorem also has applications to pseudorandomness, e.g.~\cite{meka_pseudorandom_2010,gopalan_pseudorandom_2011}. Further connections of our work to pseudorandomness are discussed in Section~\ref{sec:priorpseudo} below.

The following two sections place our results in context within the literature. While our main result is a Berry-Esseen bound (or viewed alternatively, a regularity lemma), our proof techniques evolved from a line of work on the pseudorandomness of expander walks \cite{guruswami_pseudobinomiality_2021,cohen_expander_2021,cohen_expander_2022,golowich_pseudorandomness_2022-1}, and our work also has implications for this area. We therefore first compare our result to prior Berry-Esseen bounds, and then describe the pseudorandomness implications.

\subsection{Comparison to prior Berry-Esseen bounds}
Recall that the classic Berry-Esseen Theorem \cite{berry_accuracy_1941,esseen_liapunov_1942} states that the sum of $t$ independent random variables with bounded 2nd and 3rd moments converges in Kolmogorov distance to a normal distribution, with the error decaying at a rate of $O(1/\sqrt{t})$. Markov chain Berry-Esseen theorems show that the same rate of convergence holds when the variables may be correlated according to a Markov chain. Various versions of these results have been shown under different convergence metrics and conditions on the variables and the Markov chain.


In particular, there are known Markov chain Berry-Esseen theorems that show convergence in Kolmogorov distance to a normal distribution at a rate of $O(1/\sqrt{t})$ (e.g.~\cite{bolthausen_berry-esseen_1980,bolthausen_berry-esseen_1982,herve_nagaev-guivarch_2010,kloeckner_effective_2019}).
Some of these results (e.g.~\cite{bolthausen_berry-esseen_1980,kloeckner_effective_2019}) do apply to Markov chains with a discrete state space, as in our setting. Yet convergence in Kolmogorov distance is weaker than convergence in total variation distance. Indeed, no sequence of discrete random variables can converge to a (continuous) normal distribution in total variation distance.

Berry-Esseen bounds in total variation distance for discrete random variables have been shown when the variables are not correlated according to a Markov chain. In this case, the approximating distribution is taken to be some discrete analogue of the normal distribution, such as a compound Poisson \cite{barbour_poisson_1999}, translated Poisson \cite{rollin_translated_2007}, binomial \cite{rollin_symmetric_2006}, and a histogram discretization of the continuous normal \cite{fang_discretized_2014}. While these works primarily consider the iid case, some of them, such as R\"{o}llin \cite{rollin_translated_2007,rollin_symmetric_2006} as well as Barbour and Xia \cite{barbour_poisson_1999}, also provide total variation Berry-Esseen bounds for integer-valued random variables that may have some dependencies among the variables. However, the permitted dependencies among the variables in such prior work are not general enough to apply in our setting of random walks on expanders.

In contrast, our expander walk Berry-Esseen bound applies when the variables $\val(\RW_G^t)_i$ are correlated according to a random walk on a $\lambda$-spectral expander $G$, which is a Markov chain with spectral gap $1-\lambda$. Our total variation bound of $\lambda/t^{1/2-o(1)}$ for the rate of convergence unifies the ordinary $O(1/\sqrt{t})$ Berry-Esseen convergence rate (up to a $t^{o(1)}$ loss) with the linear dependence on spectral expansion $\lambda$ from \cite{cohen_expander_2022,golowich_pseudorandomness_2022-1} (see below). As described above, to the best of our knowledge, no such bound of almost $\lambda/\sqrt{t}$ was previously known, even for the weaker Kolmogorov distance metric.

\subsection{Implications for the pseudorandomness of expander walks}
\label{sec:priorpseudo}
Our Berry-Essen bound contributes to a recent line of work \cite{guruswami_pseudobinomiality_2021,cohen_expander_2021,cohen_expander_2022,golowich_pseudorandomness_2022-1} studying the extent to which random walks on expander graphs fool symmetric functions. This problem was motivated as a generalization of the observation that expander walks fool the parity function, which plays a central role in Ta-Shma's breakthrough construction of almost-optimal $\epsilon$-balanced codes \cite{ta-shma_explicit_2017}.

Let $G$ be a $\lambda$-spectral expander with vertex labeling $\val:V(G)\rightarrow\{0,1\}$ that assigns label $b\in\{0,1\}$ to $p_b$-fraction of the vertices. The length-$t$ random walk on $G$ is said to \textit{$\epsilon$-fool} a given function $f$ on $\{0,1\}^t$ if
\begin{equation*}
  \disTV(f(\val(\RW_G^t)),f(\val(\RW_J^t)))\leq\epsilon,
\end{equation*}
where $J$ denotes the complete graph with self loops on vertex set $V(G)$ (and we extend $\val$ to act on sequences component-wise).

The recent work \cite{guruswami_pseudobinomiality_2021,cohen_expander_2021,cohen_expander_2022,golowich_pseudorandomness_2022-1} studied the problem of bounding the extent to which expander walks fool symmetric functions $f$. As a general symmetric function on $\{0,1\}^t$ only depends on the sum of the $t$ input bits, it is sufficient to consider $f(a)=\sum_{i\in[t]}a_i$, and then to bound $\disTV(\Sigma\val(\RW_G^t),\Sigma\val(\RW_J^t))$. Here by definition $\Sigma\val(\RW_J^t)=\Bin(t,p_1)$ is simply the binomial distribution.

Cohen et al.~\cite{cohen_expander_2022} and Golowich and Vadhan~\cite{golowich_pseudorandomness_2022-1} showed that for every $\lambda$-spectral expander $G$, the random walk $O(\lambda)$-fools all symmetric functions, that is,
\begin{equation}
  \label{eq:foolsym}
  \disTV\left(\Sigma\val(\RW_G^t),\Bin(t,p_1)\right) \leq O(\lambda).
\end{equation}
They also showed that this bound is tight, in the sense that there exist $\lambda$-spectral expanders $G$ for which the left hand side above is $\Omega(\lambda)$. Specifically, Golowich and Vadhan~\cite{golowich_pseudorandomness_2022-1} show that~(\ref{eq:foolsym}) is tight for the sticky walk $G=G_{\lambda,p}$ because as $t\rightarrow\infty$, the central limit theorem implies that $\Sigma\val(\RW_{G_{\lambda,p}})$ and $\Bin(t,p_1)$ converge in Kolmogorov distance to normal distributions whose variances have ratio $1+\Theta(\lambda)$. The main idea of our Berry-Esseen result is to leverage this insight to improve the bound in~(\ref{eq:foolsym}) to $O(\lambda/t^{1/2-o_p(1)})$, by replacing the binomial approximating distribution with a discrete normal of appropriate variance.

By improving upon the $O(\lambda)$ bound in~(\ref{eq:foolsym}), our Berry-Esseen bound in Theorem~\ref{thm:mainresultinf} also helps characterize which symmetric functions $f$ are $\epsilon$-fooled by expander walks for $\epsilon\ll O(\lambda)$. Specifically, our result implies that for all $\epsilon\geq\lambda/t^{1/2-o_p(1)}$, a symmetric function $f$ is $O(\epsilon)$-fooled by the random walk on an arbitrary $\lambda$-spectral expander if and only if $f$ is $O(\epsilon)$-fooled by the sticky walk $G_{\nu,p}$ for an appropriate choice of $\nu$.

This result helps explain the previously known fact that certain symmetric functions are $\lambda^{O(1)}/\sqrt{t}$-fooled by expander walks. For instance, Cohen et al.~\cite{cohen_expander_2021,cohen_expander_2022} showed that expander walks $O_p(\lambda/\sqrt{t})$-fool all the indicator functions $f(a)=\1_{\sum_ia_i=j}$ for $0\leq j\leq t$, and $O_p(\lambda^2/\sqrt{t})$-fool the threshold function $f(a)=\1_{\sum_ia_i\geq p_1t}$. Our Berry-Esseen result shows that (slightly weaker versions of) these bounds for general $G$ are implied by the respective bounds on the sticky walk.

Our Berry-Esseen result does \textit{not} explain why some symmetric functions, such as the parity function, are $e^{-\Omega(t)}$-fooled by expander walks. Rather, this exponentially small error for the parity function follows from the more general fact that the distribution $\Sigma\val(\RW_G^t)$ is smooth, in the sense that it has rapidly decaying Fourier tails; see Lemma~\ref{lem:smooth} (shown implicitly in \cite{golowich_pseudorandomness_2022-1}) and the surrounding discussion.

\begin{remark}
  In recent independent and concurrent work (posted after the submission of the undergraduate thesis containing our results \cite{golowich_random_2022}, and after the submission of this paper to STOC 2023, but before the posting of this paper online), Chiclana and Peres \cite{chiclana_local_2022} showed a local central limit theorem for expander walks, which in particular implies that $\Sigma\val(\RW_G^t)$ converges in total variation distance to an appropriate discrete normal distribution as $t\rightarrow\infty$ for any fixed graph $G$. However, Chiclana and Peres \cite{chiclana_local_2022} do not obtain a bound on the rate of this convergence, whereas our main result bounds the total variation distance by $O(\lambda/t^{1/2-o_p(1)})$ uniformly over all graphs $G$.
\end{remark}

\subsection{Open questions}
Our results lead to the following questions.
\begin{itemize}
\item Can the $t^{o_p(1)}$ factor in our bound~(\ref{eq:mainresultinf}) be removed?
\item As a regularity lemma, our result states that for every $\lambda$-spectral expander $G$, there exists a 2-state Markov chain $G'$ such that $\Sigma\val(\RW_G^t)$ is approximated by $\Sigma\val(\RW_{G'}^t)$ up to a $O(\lambda/t^{1/2-o_p(1)})$ total variation error. Can better approximations be achieved by $k$-state Markov chains $G'$ for $k>2$?
\item Are there alternative constructions of discrete normal distributions for which our bound holds? As described in Section~\ref{sec:mainresultinf}, we prove Theorem~\ref{thm:mainresultinf} for any distribution $\cN_{\sigma^2}^t$ satisfying a set of axioms, and then provide an instantiation of such a distribution $\cN_{\sigma^2}^t$ using the sticky walk. While this instantiation has the advantage of providing a regularity lemma, the sticky walk is itself nontrivial to analyze, as was the focus of Guruswami and Kumar~\cite{guruswami_pseudobinomiality_2021}. Therefore additional constructions of discrete normals satisfying our axioms would also be of interest.
\item Can Theorem~\ref{thm:mainresultinf} be extended to more general labelings $\val:V(G)\rightarrow[d]$ for $d>2$? Specifically, Golowich and Vadhan~\cite{golowich_pseudorandomness_2022-1} show a generalization of~(\ref{eq:foolsym}) for such $d$-ary labelings, so does an analogous generalization of Theorem~\ref{thm:mainresultinf} hold?
\end{itemize}

\subsection{Organization}
The remainder of this paper is organized as follows. Section~\ref{sec:prelim} describes necessary background and preliminary results. Section~\ref{sec:resultstatement} provides the formal statement of our main Berry-Esseen bound. Section~\ref{sec:proofoverview} outlines the proof of the main result, and Section~\ref{sec:beproof} presents the complete proof.

\section{Preliminaries}
\label{sec:prelim}
In this section, we describe the necessary background to present our Berry-Esseen bound. Section~\ref{sec:notation} provides the basic notation and problem setup. In Section~\ref{sec:gvpseudo}, we describe the result of Golowich and Vadhan~\cite{golowich_pseudorandomness_2022-1} that bounds how the distribution $\Sigma\val(\RW_{\cG}^t)$ changes when the expanders at some steps in $\cG$ are changed. Our proofs rely on this bound, while our main result strengthens certain implications of it.

In Sections~\ref{sec:asympvar}--\ref{sec:discnormconst}, we describe the notion of asymptotic variance as well as some basic properties, and we introduce the family of discrete normals $\cN_{\sigma^2}^t$ that we use to approximate $\Sigma\val(\RW_G^t)$ for a $\lambda$-spectral expander $G$. The proofs of results in these sections are standard or follow directly from prior work, and for completeness are provided in Appendix~\ref{app:beproofs}.

\subsection{Notation and problem setup}
\label{sec:notation}
This section introduces the basic notation and problem setup of this paper.

We use the following notation throughout. For $N\in\bN$, let $[N]=\{0,\dots,N-1\}$. For a matrix $A\in\bF^{N\times N}$, the spectral norm of $A$ is defined to be $\|A\|=\max_{x\in\bF^N\setminus\{0\}}\|Ax\|/\|x\|$. A matrix $W\in[0,1]^{N\times N}$ is a random walk matrix on $N$ vertices if the columns of $W$ sum to $1$, so that $W_{j,i}$ denotes the transition probability from vertex $i$ to vertex $j$. When the dimension $N$ is clear from context, let $I\in\bR^{N\times N}$ denote the identity matrix. Let $\vec{1}\in\bR^N$ denote the unit vector with all entries equal to $1/\sqrt{N}$, and let $J=\vec{1}\vec{1}^\top\in\bR^{N\times N}$ denote the matrix with all entries equal to $1/N$. Therefore $J$ is the random walk matrix for the $N$-vertex complete graph with self-loops.

For a regular digraph $G=(V,E)$ on $n$ vertices, the spectral expansion is defined as
\begin{equation*}
  \lambda(G) = \|G|_{\vec{1}^\perp}\| = \max_{x\perp\vec{1}}\frac{\|Gx\|}{\|x\|},
\end{equation*}
where by abuse of notation $G\in\bR^{V\times V}$ also denotes the random walk matrix of the graph $G$, so that $G_{v',v}=w_G(v,v')/\deg_G(v)$.

Given $t\in\bN$ and a sequence of random walk matrices $\cW=(W_1,\dots,W_{t-1})$ on shared vertex set $V$, let $\RW_{\cW}^t$ denote the probability distribution over $V^t$ obtained by taking a $t$-step random walk on $V$, where the $i$th step is taken according to the transition probabilities in $W_i$. Formally, to sample $(v_0,\dots,v_{t-1})\sim\RW_{\cW}^t$, the initial vertex $v_0\in V$ is chosen uniformly at random, and then for $1\leq i\leq t-1$ the vertex $v_i$ is sampled given $v_{i-1}$ according to $\Pr[v_i=v]=(W_i)_{v,v_{i-1}}$. If all $W_i$ equal some matrix $W$, we let $\RW_W^t=\RW_{\cW}^t$.

Let $G=(V,E)$ be a $\lambda$-spectral expander with some vertex labeling $\val:V\rightarrow\{0,1\}$, and let $p_b=|\val^{-1}(b)|/n$, so that $p=(p_0,p_1)$ gives the probability distribution of the label of a uniformly random vertex. For $t\in\bN$, we extend the label function component-wise to $\val:V^t\rightarrow\{0,1\}^t$. Let $\Sigma\val(\RW_G^t)=\sum_{i\in[t]}\val(\RW_G^t)_i$ denote the sum of the labels from a length-$t$ random walk on $G$.

Our goal in this paper is to study the distribution $\Sigma\val(\RW_G^t)$. Specifically, our main result shows that $\Sigma\val(\RW_G^t)$ is approximated by an appropriate discrete normal distribution up to a $O(\lambda/t^{1/2-o_p(1)})$ error in total variation distance.

\subsection{Expander walks $O(\lambda)$-fool symmetric functions}
\label{sec:gvpseudo}
As described in Section~\ref{sec:priorpseudo}, our Berry-Esseen result can be viewed as a strengthening of the result of \cite{cohen_expander_2022,golowich_pseudorandomness_2022-1} that random walks on $\lambda$-spectral expanders $O(\lambda)$-fool symmetric functions. Golowich and Vadhan~\cite{golowich_pseudorandomness_2022-1} in fact showed the following more general result, which we apply in our proof.

\begin{theorem}[\cite{golowich_pseudorandomness_2022-1}]
  \label{thm:difftail}
  Fix integers $t\geq 1$ and $1\leq u\leq t-1$. Let $\cG=(G_i)_{1\leq i\leq t-1}$ and $\cG'=(G_i')_{1\leq i\leq t-1}$ be sequences of regular graphs on a shared vertex set $V$ such that for all $i\neq u$ we have $G_i=G_i'$ with $\lambda(G_i)=\lambda(G_i')\leq 1/100$. Fix a labeling $\val:V\rightarrow\{0,1\}$ that assigns each label $b\in\{0,1\}$ to $p_b$-fraction of the vertices. Then for every $c\geq 0$,
  \begin{align}
    \label{eq:tailbound}
    \begin{split}
      \hspace{1em}&\hspace{-1em}\sum_{j\in[t+1]:|j-p_1t|\geq c}\left|\Pr[\Sigma\val(\RW_{\cG'}^t)=j]-\Pr[\Sigma\val(\RW_{\cG}^t)=j]\right| \\
      &\leq 4000 \cdot \frac{\|G_u'-G_u\| \cdot e^{-c^2/8t}}{t}.
    \end{split}
  \end{align}
\end{theorem}

The corollary below follows by changing the graphs in the sequence $\cG$ to $J$ one at a time, and applying Theorem~\ref{thm:difftail} at each step.

\begin{corollary}[\cite{golowich_pseudorandomness_2022-1}]
  \label{cor:difftailJ}
  Fix an integer $t\geq 1$. Let $\lambda\leq 1/100$, and let $\cG=(G_i)_{1\leq i\leq t-1}$ be a sequence of regular $\lambda$-spectral expanders on shared vertex set $V$. Fix a labeling $\val:V\rightarrow\{0,1\}$ that assigns each label $b\in\{0,1\}$ to $p_b$-fraction of the vertices. Then for every $c\geq 0$,
  \begin{align*}
    \begin{split}
      \hspace{1em}&\hspace{-1em}\sum_{j\in[t+1]:|j-p_1t|\geq c}\left|\Pr[\Sigma\val(\RW_{\cG}^t)=j]-\Pr[\Sigma\val(\RW_{J}^t)=j]\right| \\
      &\leq 4000 \cdot \lambda \cdot e^{-c^2/8t}.
    \end{split}
  \end{align*}
\end{corollary}

Letting $c=0$ in Corollary~\ref{cor:difftailJ} gives the total variation bound~(\ref{eq:foolsym}) described in Section~\ref{sec:priorpseudo}, which shows that expander walks $O(\lambda)$-fool symmetric functions. The $c>0$ case of Theorem~\ref{thm:difftail} and Corollary~\ref{cor:difftailJ} unites this total variation bound with a tail bound, which strengthens the expander walk Chernoff bound as described in \cite{golowich_pseudorandomness_2022-1}. We apply this tail bound in our proofs to show that we may focus on bounding components of the relevant distributions that lie near the mean.

As described in Section~\ref{sec:proofoverview} and Section~\ref{sec:beproof}, the main idea in the proof of our Berry-Esseen bound is to induct on $t$ using Theorem~\ref{thm:difftail}. Specifically, Theorem~\ref{thm:difftail} implies that if we split a length-$t$ random walk on $G$ into a sequence of $\ell$ length-$t/\ell$ independent walks on $G$ by replacing $\ell-1$ of the steps on $G$ with steps on $J$, then the distribution of the sum of the labels is perserved up to a $O(\ell\cdot\lambda/t)$ total variation error. Thus setting $\ell=\sqrt{t}$, we reduce the study of $\Sigma\val(\RW_G^t)$ to $\Sigma\val(\RW_G^{\sqrt{t}})$, from which we apply induction.

Our proof will also use the following ``smoothness'' lemma that is shown implicitly by Golowich and Vadhan~\cite{golowich_pseudorandomness_2022-1} in their proof of Theorem~\ref{thm:difftail}.

\begin{lemma}
  \label{lem:smooth}
  For $\lambda\leq 1/100$, let $G=(V,E)$ be a $\lambda$-spectral expander with vertex labeling $\val:V\rightarrow\{0,1\}$ that assigns each label $b\in\{0,1\}$ to $p_b$-fraction of the vertices. Then for all $-\pi\leq\theta\leq\pi$,
  \begin{equation*}
    |\bE[e^{-i\theta\Sigma\val(\RW_G^t)}]| \leq e^{-p_0p_1t\theta^2/20}.
  \end{equation*}
\end{lemma}

Lemma~\ref{lem:smooth} implies that the distribution of $\Sigma\val(\RW_G^t)$ has rapidly decaying Fourier tails. For instance, setting $\theta=\pi$ in this lemma recovers the previously known fact that expander walks $e^{-\Omega_p(t)}$-fool the parity function.

\subsection{Asymptotic variance}
\label{sec:asympvar}
This section describes asymptotic variance. The results below are standard; proofs are provided in Appendix~\ref{app:beproofs} for completeness.

\begin{definition}
  For a sequence $X=(X^t)^{t\in\bN}$ of probability distributions over $\bR$, the \textbf{asymptotic variance} of $X$ is
  \begin{align*}
    \sigma^2(X)
    &= \lim_{t\rightarrow\infty}\frac{1}{t}\Var(X^t).
  \end{align*}
\end{definition}

The following formula for the asymptotic variance $\sigma^2(\Sigma\val(\RW_G^t))$ is well known; it for instance is a special case of the definition of asymptotic variance in Kloeckner~\cite{kloeckner_effective_2019}.

\begin{lemma}
  \label{lem:asympvardef}
  Let $G=(V,E)$ be a regular graph with labeling $\val:V\rightarrow\{0,1\}$ that assigns each label $b\in\{0,1\}$ to $p_b$-fraction of the $n$ vertices. Viewing $\val$ and $\val-p_1:V\rightarrow\bR$ as vectors in $\bR^V$, then
  \begin{equation*}
    \sigma^2(\Sigma\val(\RW_G^t)) = p_0p_1 + 2\sum_{i=1}^\infty \frac{1}{n}(\val-p_1)^\top G^i \val.
  \end{equation*}
\end{lemma}
\begin{corollary}
  \label{cor:expasympvar}
  For a $\lambda$-spectral expander $G$,
  \begin{equation*}
    |\sigma^2(\Sigma\val(\RW_G^t)) - p_0p_1| \leq \frac{2}{1-\lambda} \cdot \lambda \cdot p_0p_1.
  \end{equation*}
\end{corollary}

For a $\lambda$-spectral expander $G$, the following lemma shows a $O(\lambda p_0p_1/t)$ bound on the rate of convergence of $\Var(\Sigma\val(\RW_G^t))/t$ to $\sigma^2(\Sigma\val(\RW_G^t))$.

\begin{lemma}
  \label{lem:asympvarconv}
  Let $G=(V,E)$ be a regular $\lambda$-spectral expander with labeling $\val:V\rightarrow\{0,1\}$ that assigns each label $b\in\{0,1\}$ to $p_b$-fraction of the $n$ vertices. Then
  \begin{equation*}
    \left|\frac{1}{t}\Var(\Sigma\val(\RW_G^t)) - \sigma^2(\Sigma\val(\RW_G^t))\right| \leq \frac{2}{(1-\lambda)^2}\cdot\frac{\lambda}{t} \cdot p_0p_1.
  \end{equation*}
\end{lemma}

\subsection{Sticky random walk}
\label{sec:sticky}
In this section, we introduce the $\lambda$-sticky, $p$-biased random walk, which is a particularly simple random walk with spectral expansion $\lambda$ and label weights $p$. We will use the sticky walk to construct discrete normal distributions, and therefore our Berry-Esseen bound can be viewed as a regularity lemma that reduces the study of arbitrary expanders to the much simpler class of sticky walks.

The special case of the sticky walk on $|V|=2$ vertices with $p_0=p_1=1/2$ was studied extensively by Guruswami and Kumar~\cite{guruswami_pseudobinomiality_2021}. Here we describe a more general sticky walk for arbitrary $p$.

\begin{definition}
  \label{def:sticky}
  Fix a vertex set $V=V_0\sqcup V_1$ with labeling $\val:V\rightarrow\{0,1\}$ given by $\val(v)=b$ for $v\in V_b$, so that $p_0=|V_0|/|V|$ and $p_1=|V_1|/|V|$. For subsets $A,B\subseteq V$, let $J_{A,B}\in\bR^{A\times B}$ denote the matrix with all entries equal to $1/|A|$. For $0\leq\lambda\leq 1$, define the \textbf{$\lambda$-sticky, $p$-biased random walk matrix $G_{\lambda,p}\in\bR^{V\times V}$} by
  \begin{equation*}
    G_{\lambda,p} = \begin{pmatrix}
      (p_0+p_1\lambda)J_{V_0,V_0} & (p_0-p_0\lambda)J_{V_0,V_1} \\
      (p_1-p_1\lambda)J_{V_1,V_0} & (p_1+p_0\lambda)J_{V_1,V_1}
    \end{pmatrix}.
  \end{equation*}
  That is, $G_{\lambda,p}$ treats all vertices within $V_b$ identically for each $b=0,1$, and if $(v,v')$ represents a 1-step random walk on $G_{\lambda,p}$, then the transition probabilities are
  \begin{align*}
    \Pr[v'\in V_0|v\in V_0] &= p_0+p_1\lambda = (1-\lambda)p_0+\lambda \\
    \Pr[v'\in V_0|v\in V_1] &= p_0-p_0\lambda = (1-\lambda)p_0 \\
    \Pr[v'\in V_1|v\in V_0] &= p_1-p_1\lambda = (1-\lambda)p_1 \\
    \Pr[v'\in V_1|v\in V_1] &= p_1+p_0\lambda = (1-\lambda)p_1+\lambda.
  \end{align*}
\end{definition}

We show that the $\lambda$-sticky random walk is indeed a $\lambda$-spectral expander.

\begin{lemma}
  \label{lem:stickyexp}
  $\lambda(G_{\lambda,p})=\lambda$.
\end{lemma}
\begin{proof}
  By definition
  \begin{equation*}
    G_{\lambda,p} = (1-\lambda)J_{V,V}+\lambda W
  \end{equation*}
  for 
  \begin{equation*}
    W = \begin{pmatrix}J_{V_0,V_0}&0\\0&J_{V_1,V_1}\end{pmatrix}.
  \end{equation*}
  We have $\|W\|=1$, as $W$ acts as $J$ on the orthogonal subspaces $\bR^{V_0}$ and $\bR^{V_1}$ of $\bR^V$. Thus $\lambda(G_{\lambda,p})\leq\lambda$. The opposite inequality follows from the fact that $p_1\1_{V_0}-p_0\1_{V_1}\in\vec{1}^\perp$ is an eigenvector of $G_{\lambda,p}$ with eigenvalue $\lambda$, as is evident from the decomposition of $G_{\lambda,p}$ above.
\end{proof}

Below, we compute the asymptotic variance of the sticky walk.

\begin{lemma}
  \label{lem:stickyvar}
  Define $G_{\lambda,p}$ and $\val$ as in Definition~\ref{def:sticky}. Then
  \begin{equation*}
    \sigma^2(\Sigma\val(\RW_{G_{\lambda,p}}^t)) = p_0p_1 \cdot \frac{1+\lambda}{1-\lambda}.
  \end{equation*}
\end{lemma}
\begin{proof}
  View $\val$ and $\val-p_1:V\rightarrow\bR$ as vectors in $\bR^V$. Then $\val-p_1\in\vec{1}^\perp$ is an eigenvalue of $G_{\lambda,p}=(1-\lambda)J_{V,V}+\lambda\begin{pmatrix}J_{V_0,V_0}&0\\0&J_{V_1,V_1}\end{pmatrix}$ with eigenvalue $\lambda$, so by Lemma~\ref{lem:asympvardef},
  \begin{align*}
    \sigma^2(\Sigma\val(\RW_{G_{\lambda,p}}^t))
    &= p_0p_1 + 2\sum_{i=1}^\infty \frac{1}{n}(\val-p_1)^\top G^i (\val-p_1) \\
    &= p_0p_1 + 2\sum_{i=1}^\infty \lambda^i \cdot \frac{\|\val-p_1\|^2}{n} \\
    &= p_0p_1 + 2 \cdot \frac{\lambda}{1-\lambda} \cdot p_0p_1 \\
    &= p_0p_1 \cdot \frac{1+\lambda}{1-\lambda}.
  \end{align*}
\end{proof}

\subsection{Definition of discrete normal families}
\label{sec:discnormdef}
To formally state our Berry-Esseen bound, we must define the family of discrete normal approximating distributions $\cN_{\sigma^2}^t$. We take an axiomatic approach here, where we define below the set of conditions that this family of distributions must satisfy for our Berry-Esseen proof. We then prove that such a family can be constructed using the sticky random walk.

\begin{definition}
  \label{def:discnorm}
  For $p=(p_0,p_1)$, $\sigma^2>0$, and $c=(c_1,c_2,c_3)\in\bR_+^3$, a \textbf{$(p,\sigma^2,c)$-discrete normal family} is a family\footnote{In this section, we typically treat $p$ and $c$ as fixed constants, so we exclude them in the notation $\cN_{\sigma^2}^t$ for readability.} $(\cN_{\sigma^2}^t)^{t\in\bN}$ of probability distributions over $\bZ$ such that the following conditions hold for all $t\in\bN$:
  \begin{enumerate}
  \item \label{it:dnexp} $\bE[\cN_{\sigma^2}^t] = p_1t$.
  \item \label{it:dnvar} $|\Var(\cN_{\sigma^2}^t)-\sigma^2t| \leq c_1 \cdot |\sigma^2-p_0p_1|$.
  \item \label{it:dnbern} $\cN_{\sigma^2}^1 = \text{Bern}(p)$.
  \item \label{it:dnsum} For all positive integers $\ell\leq t$ and $t_0,\dots,t_{\ell-1}$ such that $\sum_{i\in[\ell]}t_i=t$, then
    \begin{equation*}
      \disTV\left(\sum_{i\in[\ell]}\cN_{\sigma^2}^{t_i},\; \cN_{\sigma^2}^t\right) \leq \frac{c_2}{2} \cdot \frac{(\ell-1)\cdot|\sigma^2/p_0p_1-1|}{t},
    \end{equation*}
    where the variables $\cN_{\sigma^2}^{t_i}$ in the sum above are independent.
  \item \label{it:dntail} For all $a\geq 0$,
    \begin{equation*}
      \sum_{j\in\bZ:|j-p_1t|\geq a}|\Pr[\cN_{\sigma^2}^t=j]-\Pr[\text{Bin}(t,p)=j]| \leq c_3 \cdot |\sigma^2/p_0p_1-1| \cdot e^{-a^2/8t}.
    \end{equation*}
  \item \label{it:dnfourier} The characteristic function of $\cN_{\sigma^2}^t$ satisfies $|\bE[e^{-i\theta\cN_{\sigma^2}^t}]| \leq e^{-p_0p_1t\theta^2/20}$.
  \end{enumerate}
\end{definition}

Note that the constants $8$ and $20$ in the exponents above could be replaced with generic constants $c_4$ and $c_5$, but we leave them as explicit values to simplify notation.

In Definition~\ref{def:discnorm}, it is helpful to think of $p$ and $c$ as fixed constants. For given $p,c$, we want to allow $\sigma^2$ to vary within a neighborhood of $p_0p_1$, in order to obtain discrete normals parametrized by mean $p_1t$ and variance $\approx\sigma^2 t$. Proposition~\ref{prop:stickydn} below shows that such a family exists, and is given by $\Sigma\val(\RW_{G_{\lambda,p}}^t)$ for the sticky random walk $G_{\lambda,p}$ (see Section~\ref{sec:sticky}) with $\lambda=(\sigma^2-p_0p_1)/(\sigma^2+p_0p_1)$.

Although Definition~\ref{def:discnorm} provides a nonstandard definition of an integer-valued discrete normal, all of the conditions are natural. Condition~\ref{it:dnexp} and condition~\ref{it:dnvar} simply specify that $\cN_{\sigma^2}^t$ has expectation $p_1t$ and variance approximately $\sigma^2t$. Condition~\ref{it:dnbern} specifies that the $t=1$ case $\cN_{\sigma^2}^1$ takes values in $\{0,1\}$, as is for instance satisfied by a binomial distribution. Condition~\ref{it:dnsum} specifies that the sum of discrete normals is approximately a discrete normal, just as the sum of continuous normals is a normal distribution. Condition~\ref{it:dntail} specifies that the tails of $\cN_{\sigma^2}^t$ approach the tails of the binomial, which is natural as the binomial is a canonical discrete normal. Condition~\ref{it:dnfourier} specifies that the Fourier tails of the discrete normal decay rapidly, as is again the case for continuous normals.

\subsection{Construction of discrete normal families}
\label{sec:discnormconst}
This section describes our construction of discrete normal families using the sticky walk.

Note that while Definition~\ref{def:sticky} only defined the sticky random walk for $0\leq\lambda\leq 1$, the definition extends naturally to all $-\min\{p_0/p_1,p_1/p_0\}\leq\lambda\leq 1$. Many properties from the $\lambda\geq 0$ case extend to the $\lambda<0$ case. In particular, for all $-\min\{p_0/p_1,p_1/p_0\}\leq\lambda\leq 1$, the $\lambda$-sticky walk has spectral expansion $\lambda(G_{\lambda,p})=|\lambda|$, as can be seen from Lemma~\ref{lem:stickyexp} along with the fact that $G_{-\lambda,p}=2J-G_{\lambda,p}$. Our proof of Lemma~\ref{lem:stickyvar} also still holds for all $-\min\{p_0/p_1,p_1/p_0\}\leq\lambda\leq 1$.

For $\lambda<0$, the $\lambda$-sticky walk can be thought of as a ``$(-\lambda)$-jumpy'' walk, as in this case the probabilities are skewed towards jumping to the opposite set $V_{1-b}=\val^{-1}(1-b)$ from $V_b=\val^{-1}(b)$. However, there is an unfortunate complication arising in this interpretation, which explains why we cannot let $\lambda$ descend all the way to $-1$ when $p\neq(1/2,1/2)$. Recall that $G_{\lambda,p}=(1-\lambda)J+\lambda\begin{pmatrix}J_{V_0,V_0}&0\\0&J_{V_1,V_1}\end{pmatrix}$, so that the $\lambda$-sticky random walk can be interpreted as going to a random vertex in $V$ with probability $1-\lambda$, and remaining in the current set $V_b$ with probability $\lambda$. This interpretation does not extend to $\lambda<0$. A more natural ``$\mu$-jumpy'' random walk, which is defined for all $0\leq\mu=-\lambda\leq 1$, would be given by the random walk matrix $(1-\mu)J+\mu\begin{pmatrix}0&J_{V_0,V_1}\\J_{V_1,V_0}&0\end{pmatrix}$, which corresponds to going to a random vertex in $V$ with probability $1-\mu$, and jumping to the opposite set $V_{1-b}$ from $V_b$ with probability $\mu$. Unfortunately, this notion of a jumpy random walk does not come from a regular graph when $p\neq(1/2,1/2)$, so we do not use it here.

The $\lambda$-sticky, $p$-biased random walk cannot be extended to $\lambda<-\min\{p_0/p_1,p_1/p_0\}$ without having negative values in the random walk matrix $G_{\lambda,p}$. For the purpose of our proofs, such negative probabilities do not seem to pose any fundamental issue, and we may in fact be able to consider $G_{\lambda,p}$ for all $-1\leq\lambda\leq 1$. However, to avoid confusion, we do not pursue this generalization.

\begin{proposition}
  \label{prop:stickydn}
  For every $p=(p_0,p_1)$ and every
  \begin{equation*}
    \frac{1-\min\{\frac{p_0}{p_1},\frac{p_1}{p_0},\frac{1}{100}\}}{1+\min\{\frac{p_0}{p_1},\frac{p_1}{p_0},\frac{1}{100}\}} \cdot p_0p_1 \leq \sigma^2 \leq \frac{1+\frac{1}{100}}{1-\frac{1}{100}} \cdot p_0p_1,
  \end{equation*}
  there exists a $(p,\sigma^2,c)$-discrete normal family $(\cN_{\sigma^2}^t)^{t\in\bN}$ with
  \begin{align*}
    c_1 &= 2 \\
    c_2 &= 2020 \\
    c_3 &= 2020.
  \end{align*}
  Specifically, such a family is given by
  \begin{align*}
    \cN_{\sigma^2}^t &= \Sigma\val(\RW_{G_{\lambda,p}}^t),
  \end{align*}
  where $G_{\lambda,p}$ denotes the $\lambda$-sticky, $p$-biased random walk with $\lambda=(\sigma^2-p_0p_1)/(\sigma^2+p_0p_1)$.
\end{proposition}

The proof of Proposition~\ref{prop:stickydn} consists of straightforward applications of results that were stated above; it is provided in Appendix~\ref{app:beproofs}.

\section{Statement of Berry-Esseen bound}
\label{sec:resultstatement}
We are now ready to formally state our main Berry-Esseen result.

\begin{theorem}
  \label{thm:berryesseen}
  Fix $p=(p_0,p_1)$. For $\lambda\leq 1/100$, let $G=(V,E)$ be a $\lambda$-spectral expander with labeling $\val:V\rightarrow\{0,1\}$ that assigns each label $b\in\{0,1\}$ to $p_b$-fraction of the vertices. Let $\sigma^2=\sigma^2(\Sigma\val(\RW_G^t))$ denote the asymptotic variance. For some $c\in\bR_+^3$, let $(\cN_{\sigma^2}^t)^{t\in\bN}$ be a $(p,\sigma^2,c)$-discrete normal family. Then for all $t\in\bN$,
  \begin{equation*}
    \disTV\left(\Sigma\val(\RW_G^t),\; \cN_{\sigma^2}^t\right) \leq \frac{\lambda}{\sqrt{t}} \cdot (1+\log t)^{\eta_1 \log\log t + \eta_2},
  \end{equation*}
  where
  \begin{align*}
    \eta_1 &= 140 \\
    \eta_2 &= 140 + 3 \cdot \log\left(\frac{2^{28}+2^{10}c_1+2^{18}c_3}{(p_0p_1)^{7/2}} + 3c_2\right).
  \end{align*}
\end{theorem}

We give explicit constants $\eta_1,\eta_2$ for the sake of completeness, but we do not attempt to optimize these constants.

Recall that Corollary~\ref{cor:difftailJ} shown by \cite{golowich_pseudorandomness_2022-1} implies that $\disTV\left(\Sigma\val(\RW_G^t),\; \Sigma\val(\RW_J^t)\right)=O(\lambda)$, while condition~\ref{it:dntail} of Definition~\ref{def:discnorm} along with Corollary~\ref{cor:expasympvar} implies that $\disTV\left(\cN_{\sigma^2}^t,\; \Sigma\val(\RW_J^t)\right)=O(\lambda)$. Thus the results of Golowich and Vadhan~\cite{golowich_pseudorandomness_2022-1} imply that $\disTV\left(\Sigma\val(\RW_G^t),\; \cN_{\sigma^2}^t\right)=O(\lambda)$. Theorem~\ref{thm:berryesseen} improves this bound to $O(\lambda/t^{1/2-o(1)})$, at the cost of a worse dependence on the label weights $p$.



\section{Proof overview}
\label{sec:proofoverview}
This section outlines the proof of Theorem~\ref{thm:berryesseen}. The full proof is given in Section~\ref{sec:beproof} below. At a high level, we follow the standard proof of the central limit theorem, in that we prove our bound using a Taylor approximation of the characteristic function of $\Sigma\val(\RW_G^t)$. However, we bound the third moment of this random variable using induction on the walk length $t$, which allows us to obtain the desired linear dependence on the spectral expansion $\lambda$ in our bound. The $(1+\log t)^{O(\log\log t)}=t^{o(1)}$ factor arises from a $\poly\log t$ loss at each step of the induction. It is an interesting question whether a tighter analysis could remove this factor.

We now present the inductive argument. Throughout this section, we take the label weighting $p$ to be a fixed constant. Our goal is to prove that for a sufficiently large constant $\eta$,
\begin{equation}
  \label{eq:beinf}
  \disTV\left(\Sigma\val(\RW_G^t),\; \cN_{\sigma^2}^t\right) \leq \frac{\lambda}{\sqrt{t}} \cdot (1+\log t)^{\eta\log\log t}.
\end{equation}
We will prove this inequality by induction over $t$. The base case of $t=1$ is immediate from the definition of $\cN_{\sigma^2}^1$. For the inductive step, assume that~(\ref{eq:beinf}) holds for walks of length $\sqrt{t}$ (assuming for simplicity that $t$ is a perfect square to avoid rounding issues).

To begin, we split the length-$t$ walk into $\sqrt{t}$ walks of length $\sqrt{t}$. The sum $\sum_{k\in[\sqrt{t}]}\Sigma\val(\RW_G^{\sqrt{t}})$ of $\sqrt{t}$ independent variables can be expressed as $\Sigma\val(\RW_{\cG'}^t)$ where $\cG'=(G'_1,\dots,G'_{t-1})$ consists of $\sqrt{t}-1$ evenly spaced copies of $J$ among $t-\sqrt{t}$ copies of $G$. Thus because $\|G-J\|=\lambda(G)\leq\lambda$, Theorem~\ref{thm:difftail} implies that
\begin{align*}
  \disTV\left(\Sigma\val(\RW_G^t),\; \sum_{k\in[\sqrt{t}]}\Sigma\val(\RW_G^{\sqrt{t}})\right) &= O\left(\frac{\lambda}{\sqrt{t}}\right).
\end{align*}
Similarly, Definition~\ref{def:discnorm} implies that
\begin{align*}
  \disTV\left(\cN_{\sigma^2}^t,\; \sum_{k\in[\sqrt{t}]}\cN_{\sigma^2}^{\sqrt{t}}\right) &=  O\left(\frac{\lambda}{\sqrt{t}}\right).
\end{align*}
Thus by the triangle inequality, to show~(\ref{eq:beinf}), it is sufficient to show that
\begin{align*}
  \disTV\left(\sum_{k\in[\sqrt{t}]}\Sigma\val(\RW_G^{\sqrt{t}}),\; \sum_{k\in[\sqrt{t}]}\cN_{\sigma^2}^{\sqrt{t}}\right) &\leq \frac{\lambda}{\sqrt{t}} \cdot (1+\log t)^{\eta\log\log t-\Omega(1)}.
\end{align*}
By the tail bounds in Theorem~\ref{thm:difftail} and Definition~\ref{def:discnorm}, almost all of the mass in the probability distributions above is contained within an interval of length roughly $O(\sqrt{t})$ around the mean $p_1t$, so by Cauchy-Schwartz it is in fact sufficient to show the $\ell_2$-bound
\begin{align}
  \label{eq:bel2inf}
  \disltwo\left(\sum_{k\in[\sqrt{t}]}\Sigma\val(\RW_G^{\sqrt{t}}),\; \sum_{k\in[\sqrt{t}]}\cN_{\sigma^2}^{\sqrt{t}}\right) &\leq \frac{\lambda}{t^{3/4}} \cdot (\log t)^{\eta\log\log t-\Omega(1)}.
\end{align}

We show~(\ref{eq:bel2inf}) by taking the Fourier transform, and then proving the $\ell_2$-bound using the inductive hypothesis. Specifically,
denote the probability distributions of $\Sigma\val(\RW_G^{\sqrt{t}})$ and $\cN_{\sigma^2}^{\sqrt{t}}$ by
\begin{align*}
  (h_G^{\sqrt{t}})_j &= \Pr[\Sigma\val(\RW_G^{\sqrt{t}}) = j] \\
  (n_G^{\sqrt{t}})_j &= \Pr[\cN_{\sigma^2}^{\sqrt{t}} = j],
\end{align*}
and denote their centered characteristic functions by
\begin{align*}
  \hat{h}_G^{\sqrt{t}}(\theta) &= \bE[e^{-i\theta(\Sigma\val(\RW_G^{\sqrt{t}})-p_1\sqrt{t})}] \\
  \hat{n}_G^{\sqrt{t}}(\theta) &= \bE[e^{-i\theta(\cN_{\sigma^2}^{\sqrt{t}}-p_1\sqrt{t})}].
\end{align*}
Then the left hand side of~(\ref{eq:bel2inf}) equals $\|(\hat{h}_G^{\sqrt{t}})^{\sqrt{t}}-(\hat{n}_G^{\sqrt{t}})^{\sqrt{t}}\|$ because the Fourier transform preserves $\ell_2$-norms. We will bound $\hat{h}_G^{\sqrt{t}}-\hat{n}_G^{\sqrt{t}}$ using a Taylor expansion around $\theta=0$. By definition,
\begin{align*}
  \label{eq:lowdersinf}
  \begin{split}
    (\hat{h}_G^{\sqrt{t}}-\hat{n}_G^{\sqrt{t}})(0) &= 0 \\
    {\dv{}{\theta}}(\hat{h}_G^{\sqrt{t}}-\hat{n}_G^{\sqrt{t}})(0) &= 0 \\
    \left|{\dv{^2}{\theta^2}}(\hat{h}_G^{\sqrt{t}}-\hat{n}_G^{\sqrt{t}})(0)\right|
    &= |\Var(\Sigma\val(\RW_G^{\sqrt{t}}))-\Var(\cN_{\sigma^2}^{\sqrt{t}})| = O(\lambda) \\
    \left|{\dv{^3}{\theta^3}}(\hat{h}_G^{\sqrt{t}}-\hat{n}_G^{\sqrt{t}})(\theta)\right| &\leq \sum_{j\in\bZ}|j-p_1\sqrt{t}|^3\cdot|(h_G^{\sqrt{t}})_j-(n_G^{\sqrt{t}})_j|,
  \end{split}
\end{align*}
where the bound in the third line above holds by Lemma~\ref{lem:asympvarconv}, Corollary~\ref{cor:expasympvar}, and Definition~\ref{def:discnorm}.
The key point in the proof is now to bound the third derivative above using the inductive hypothesis
\begin{align*}
  \|h_G^{\sqrt{t}}-n_G^{\sqrt{t}}\|_1 &\leq \frac{\lambda}{t^{1/4}} \cdot (1+\log\sqrt{t})^{\eta\log\log\sqrt{t}}
\end{align*}
Combining this inequality with the tail bounds in Theorem~\ref{thm:difftail} and Definition~\ref{def:discnorm}, which imply that $h_G^{\sqrt{t}}-n_G^{\sqrt{t}}$ is mostly supported within an interval of length $\tilde{O}(t^{1/4})$ around $p_1\sqrt{t}$, gives that
\begin{align*}
  \left|{\dv{^3}{\theta^3}}(\hat{h}_G^{\sqrt{t}}-\hat{n}_G^{\sqrt{t}})(\theta)\right|
  &\leq \tilde{O}(t^{1/4})^3 \cdot \|h_G^{\sqrt{t}}-n_G^{\sqrt{t}}\|_1 \\
  &\leq \lambda\sqrt{t} \cdot (1+\log t)^{\eta\log\log\sqrt{t}+O(1)}.
\end{align*}
Thus we have the Taylor approximation
\begin{align*}
  |\hat{h}_G^{\sqrt{t}}(\theta)-\hat{n}_G^{\sqrt{t}}(\theta)| &\leq \lambda\cdot(|\theta|^2+\sqrt{t}\cdot|\theta|^3)\cdot(1+\log t)^{\eta\log\log\sqrt{t}+O(1)}.
\end{align*}

Now by the Fourier tail bounds in Lemma~\ref{lem:smooth} and in Definition~\ref{def:discnorm}, both $|\hat{h}_G^{\sqrt{t}}(\theta)|$ and $|\hat{n}_G^{\sqrt{t}}(\theta)|$ are bounded above by $e^{-\Omega(\sqrt{t}\cdot\theta^2)}$. Therefore
\begin{align*}
  \left|\left(\hat{h}_G^{\sqrt{t}}(\theta)\right)^{\sqrt{t}}-\left(\hat{n}_G^{\sqrt{t}}(\theta)\right)^{\sqrt{t}}\right|
  &= \left|\sum_{k=0}^{\sqrt{t}-1}\left(\hat{n}_G^{\sqrt{t}}(\theta)\right)^k\cdot(\hat{h}_G^{\sqrt{t}}(\theta)-\hat{n}_G^{\sqrt{t}}(\theta))\cdot\left(\hat{h}_G^{t_{k'}}(\theta)\right)^{\sqrt{t}-k-1}\right| \\
  &\leq \sqrt{t}\cdot|\hat{h}_G^{\sqrt{t}}(\theta)-\hat{n}_G^{\sqrt{t}}(\theta)| \cdot e^{-\Omega(t\cdot\theta^2)} \\
  &\leq \frac{\lambda}{\sqrt{t}} \cdot (t|\theta|^2+t^{3/2}|\theta|^3) \cdot e^{-\Omega(t\theta^2)} \cdot (1+\log t)^{\eta\log\log\sqrt{t}+O(1)}.
\end{align*}

Squaring the bound above and integrating over $-\pi<\theta\leq\pi$ then gives the desired $\ell_2$-bound (\ref{eq:bel2inf}). Intuitively, for $|\theta|\leq \tilde{O}(1/\sqrt{t})$ the right hand side above is bounded by $\lambda/\sqrt{t} \cdot (1+\log t)^{\eta\log\log\sqrt{t}+O(1)}$, while for $|\theta|\gg 1/\sqrt{t}$ the right hand side above decays rapidly. Therefore the $\ell_2$-norm of this function is bounded by
\begin{align*}
  \left\|\left(\hat{h}_G^{\sqrt{t}}\right)^{\sqrt{t}}-\left(\hat{n}_G^{\sqrt{t}}\right)^{\sqrt{t}}\right\|
  &\leq \sqrt{\tilde{O}\left(\frac{1}{\sqrt{t}}\right) \cdot \left(\frac{\lambda}{\sqrt{t}} \cdot (1+\log t)^{\eta\log\log\sqrt{t}+O(1)}\right)^2} \\
  &\leq \frac{\lambda}{t^{3/4}} \cdot (1+\log t)^{\eta\log\log t-\Omega(1)},
\end{align*}
where the final inequality above assumes that $\eta$ is sufficiently large so that $\eta\log 2$ dominates the $O(1)$ constant in the exponent.
Then~(\ref{eq:bel2inf}) follows because the left hand side above equals the left hand side of~(\ref{eq:bel2inf}), as the Fourier transform interchanges convolution and multiplication, and preserves $\ell_2$-norms. Thus~(\ref{eq:beinf}) holds, completing the inductive step.

\section{Proof of Berry-Esseen bound}
\label{sec:beproof}
This section presents the full proof of Theorem~\ref{thm:berryesseen}.

  We first introduce some notation.
  Fix a graph $G$, and set $\sigma^2=\sigma^2(\Sigma\val(\RW_G^t))$ to be the asymptotic variance. For every $t\in\bN$, define vectors $h_G^t,n_G^t\in[0,1]^{\bZ}\subseteq\bR^{\bZ}$ by
  \begin{align*}
    (h_G^t)_j &= \Pr[\Sigma\val(\RW_G^t) = j] \\
    (n_G^t)_j &= \Pr[\cN_{\sigma^2}^t = j].
  \end{align*}
  That is, $h_G^t$ is the probability distribution of $\Sigma\val(\RW_G^t)$, and $n_G^t$ is the probability distribution of the discrete normal $\cN_{\sigma^2}^t$ with the same mean and asymptotic variance as $\Sigma\val(\RW_G^t)$. For a given $t$, Lemma~\ref{lem:asympvarconv}, Corollary~\ref{cor:expasympvar}, and condition~\ref{it:dnvar} of Definition~\ref{def:discnorm} imply that
  \begin{equation}
    \label{eq:vardif}
    \left|\Var(\Sigma\val(\RW_G^t)))-\Var(\cN_{\sigma^2}^t)\right| \leq \left(\frac{2}{(1-\lambda)^2}+\frac{2c_1}{1-\lambda}\right)\cdot\lambda\cdot p_0p_1.
  \end{equation}

  For a probability distribution $f^t\in[0,1]^{\bZ}$ with mean $\sum_{j\in\bZ}jf_j^t = p_1t$ (e.g.~$f^t=h_G^t$ or $n_G^t$), in this proof we define the \textbf{centered Fourier transform}
  \begin{equation*}
    \hat{f}^t(\theta) = \sum_{j\in\bZ} e^{-i\theta(j-p_1t)}f_j^t.
  \end{equation*}
  The centered Fourier transform is by definition equal to $e^{i\theta p_1t}$ times the ordinary Fourier transform. Therefore in particular, the absolute values of the centered and ordinary Fourier transforms agree, so the centered Fourier transform preserves $\ell_2$-norms like the ordinary Fourier transform. Furthermore, recall that the sum of independent random variables with distributions $f^{t_1}$ and $f^{t_2}$ has distribution given by the convolution $f^{t_1}*f^{t_2}$. As with the ordinary Fourier transform, the centered Fourier transform changes convolution into multiplication, that is, $\widehat{f^{t_1}*f^{t_2}}=\hat{f}^{t_1}\cdot\hat{f}^{t_2}$. Note that we cannot instead center the distribution $f^t$ and then take the ordinary Fourier transform because if $p_1t\notin\bZ$, then the centered version of $f^t$ no longer takes values in $\bZ$.

\begin{proof}[Proof of Theorem~\ref{thm:berryesseen}]
  We will prove by induction that for all $t\in\bN$, we have the desired inequality
  \begin{equation}
    \label{eq:beinduct}
    \|h_G^t-n_G^t\|_1 \leq \frac{\lambda}{\sqrt{t}} \cdot (1+\log t)^{\eta_1 \log\log t+\eta_2}.
  \end{equation}
  For the base case of our induction, when $t=1$, then $\Sigma\val(\RW_G^1)$ is simply a Bernoulli distribution with parameter $p$, and by condition~\ref{it:dnbern} of Definition~\ref{def:discnorm}, $\cN_{\sigma^2}^1$ is also a Bernoulli with parameter $p$. Thus $h_G^1=n_G^1$, so~(\ref{eq:beinduct}) holds for $t=1$.

  For the inductive step, fix $t\geq 1$. We will prove that the inequality~(\ref{eq:beinduct}) holds for $t$ assuming that it holds for all $u<t$. To begin, we reduce our problem of studying a random walk on $G$ of length $t$ to studying $\ell\approx\sqrt{t}$ independent random walks on $G$ of length approximately $\sqrt{t}$. We will then apply the inductive hypothesis on these shorter random walks. Formally, choose $\ell\in\bN$ and $t_0,\dots,t_{\ell-1}\in\bN$ such that $\sum_{k\in[\ell]}t_k=t$, and such that $\ell$ and all $t_k$ differ from $\sqrt{t}$ by less than $1$. In the case that $t=2$, we specify that $\ell=2$ and $t_0=t_1=1$, so that for all $t\geq 2$ we have all $t_k\leq\min\{\sqrt{t}+1,t-1\}$.

  Let $S=\{j\in\bZ:|j-p_1t|\leq 8\sqrt{t\log t}\}$. The distributions $h_G^t$ and $n_G^t$ are mostly supported on $S$, so the main point of our induction is to bound the $\ell_1$-norm of the restriction $(h_G^t-n_G^t)_S$. We will separately apply tail bounds to bound $\|(h_G^t-n_G^t)_{\bZ\setminus S}\|_1$. Specifically, we begin with the following lemma.
  
  \begin{lemma}
    \label{lem:startind}
    \begin{align*}
      \|h_G^t-n_G^t\|_1
      &\leq \left(8000 + \frac{2(c_2+c_3)}{1-\lambda}\right) \cdot \frac{\lambda}{\sqrt{t}} + \left\|\left(\bigast_{k\in[\ell]}h_G^{t_k}-\bigast_{k\in[\ell]} n_G^{t_k}\right)_S\right\|_1.
    \end{align*}
  \end{lemma}
  \begin{proof}
    Let $b^t\in[0,1]^{\bZ}$ denote the binomial distribution with parameters $t,p$. Then
    \begin{align*}
      \|h_G^t-n_G^t\|_1
      &\leq \left\|\left(h_G^t-\bigast_{k\in[\ell]} h_G^{t_k}\right)_S\right\|_1 + \left\|\left(\bigast_{k\in[\ell]} h_G^{t_k}-\bigast_{k\in[\ell]} n_G^{t_k}\right)_S\right\|_1 + \left\|\left(\bigast_{k\in[\ell]} n_G^{t_k}-n_G^t\right)_S\right\|_1 \\
      &\hspace{.5cm}+ \|(h_G^t-b^t)_{\bZ\setminus S}\|_1 + \|(b^t-n_G^t)_{\bZ\setminus S}\|_1.
    \end{align*}
    By definition $\bigast_{k\in[\ell]} h_G^{t_k}$ is the distribution of the sum $\sum_{k\in[\ell]}\Sigma\val(\RW_G^{t_k})$ of $\ell$ independent variables, which can be expressed as $\Sigma\val(\RW_{\cG'}^t)$ where $\cG'=(G'_1,\dots,G'_{t-1})$ consists of $\ell-1$ copies of $J$ and $t-\ell$ copies of $G$. Thus because $\|G-J\|=\lambda(G)\leq\lambda$, Theorem~\ref{thm:difftail} implies that
    \begin{align*}
      \left\|h_G^t-\bigast_{k\in[\ell]} h_G^{t_k}\right\|_1
      &= 2 \cdot \disTV(\Sigma\val(\RW_G^t),\; \Sigma\val(\RW_{\cG'}^t)) \\
      &\leq 4000 \cdot \frac{\lambda}{t} \cdot (\ell-1) \\
      &\leq 4000 \cdot \frac{\lambda}{\sqrt{t}}.
    \end{align*}
    Similarly, $\bigast_{k\in[\ell]} n_G^{t_k}$ is the distribution of the sum $\sum_{k\in[\ell]}\cN_{\sigma^2}^{t_k}$ of $\ell$ independent variables, so by condition~\ref{it:dnsum} of Definition~\ref{def:discnorm} along with Corollary~\ref{cor:expasympvar},
    \begin{align*}
      \left\|\bigast_{k\in[\ell]} n_G^{t_k}-n_G^t\right\|_1
      &= 2 \cdot \disTV\left(\sum_{k\in[\ell]}\cN_{\sigma^2}^{t_k},\; \cN_{\sigma^2}^t\right) \\
      &\leq c_2 \cdot \frac{(\ell-1)\cdot|\sigma^2/p_0p_1-1|}{t} \\
      &\leq c_2 \cdot \frac{\sqrt{t} \cdot \frac{2}{1-\lambda}\cdot\lambda}{t} \\
      &= \frac{2c_2}{1-\lambda} \cdot \frac{\lambda}{\sqrt{t}}.
    \end{align*}
    Furthermore, Theorem~\ref{thm:difftail} implies that
    \begin{align*}
      \|(h_G^t-b^t)_{\bZ\setminus S}\|_1
      &\leq 4000 \cdot \lambda \cdot \frac{1}{t^8},
    \end{align*}
    and condition~\ref{it:dntail} of Definition~\ref{def:discnorm} with Corollary~\ref{cor:expasympvar} implies that
    \begin{align*}
      \|(b^t-n_G^t)_{\bZ\setminus S}\|_1
      &\leq c_3 \cdot \frac{2\lambda}{1-\lambda} \cdot \frac{1}{t^8}.
    \end{align*}
    Combining the above inequalities gives the desired inequality
    \begin{align*}
      \|h_G^t-n_G^t\|_1
      &\leq \left(8000 + \frac{2(c_2+c_3)}{1-\lambda}\right) \cdot \frac{\lambda}{\sqrt{t}} + \left\|\left(\bigast_{k\in[\ell]}h_G^{t_k}-\bigast_{k\in[\ell]} n_G^{t_k}\right)_S\right\|_1.
    \end{align*}
  \end{proof}
  
  Thus our problem is reduced to bounding $\left\|\left(\bigast_{k\in[\ell]} h_G^{t_k}-\bigast_{k\in[\ell]} n_G^{t_k}\right)_S\right\|_1$. For this purpose, we will bound the $\ell_2$-norm of $\bigast_{k\in[\ell]} h_G^{t_k}-\bigast_{k\in[\ell]} n_G^{t_k}$ by bounding its centered Fourier coefficients. We will then obtain an $\ell_1$-bound from the $\ell_2$-bound by applying the Cauchy-Schwartz inequality.

  By definition, the centered Fourier transform of $\bigast_{k\in[\ell]} h_G^{t_k}-\bigast_{k\in[\ell]} n_G^{t_k}$ is given by $\prod_{k\in\ell}\hat{h}_G^{t_k}(\theta)-\prod_{k\in[\ell]}\hat{n}_G^{t_k}(\theta)$. Thus we will first bound the difference between $\hat{h}_G^{t_k}(\theta)$ and $\hat{n}_G^{t_k}(\theta)$ in the following lemma by applying the inductive hypothesis, which is the key technical step in our proof.
  
  \begin{lemma}
    \label{lem:taylorbound}
    For every $1\leq u\leq t-1$, we have
    \begin{align*}
      |\hat{h}_G^u(\theta)-\hat{n}_G^u(\theta)|
      &\leq \lambda \cdot \left(\frac{C_1}{2} \cdot |\theta|^2 + \frac{C_2}{6} \cdot u \cdot |\theta|^3\right) \cdot (1+\log u)^{\eta_1 \log\log u+\eta_2+3/2}
    \end{align*}
    for constants
    \begin{align*}
      C_1 &= \left(\frac{2}{(1-\lambda)^2}+\frac{2c_1}{1-\lambda}\right) \cdot p_0p_1 \\
      C_2 &= 2^9 + 2^8\left(4000+\frac{2c_3}{1-\lambda}\right).
    \end{align*}
  \end{lemma}
  \begin{proof}
    If $u=1$, then the desired result follows directly from the fact that $h_G^1=n_G^1$ by condition~\ref{it:dnbern} of Definition~\ref{def:discnorm}. Thus we may assume that $u\geq 2$.
    
    Recall that for $u\in\bN$, if $f^u\in[0,1]^{\bZ}\subseteq\bR^{\bZ}$ is the probability distribution for a random variable $F^u$ with mean $\bE[F^u]=p_1u$, then the centered Fourier transform $\hat{f}^u$ satisfies
    \begin{align*}
      \hat{f}^u(0) &= 1 \\
      {\dv{\hat{f}^u}{\theta}}(0) &= -i(\bE[F^u]-p_1u) = 0 \\
      {\dv{^2\hat{f}^u}{\theta^2}}(0) &= -\bE[(F^u-p_1u)^2] = -\Var(F^u) \\
      {\dv{^3\hat{f}^u}{\theta^3}}(\theta) &= \sum_{j\in\bZ}ie^{-i\theta(j-p_1t)}\cdot(j-p_1u)^3f_j^u.
    \end{align*}
    Thus for every $2\leq u\leq t-1$, we have
    \begin{align}
      \label{eq:lowders}
      \begin{split}
        (\hat{h}_G^u-\hat{n}_G^u)(0) &= 0 \\
        {\dv{}{\theta}}(\hat{h}_G^u-\hat{n}_G^u)(0) &= 0 \\
        \left|{\dv{^2}{\theta^2}}(\hat{h}_G^u-\hat{n}_G^u)(0)\right|
        &= |\Var(\Sigma\val(\RW_G^u))-\Var(\cN_{\sigma^2}^u)| \leq C_1 \cdot \lambda,
      \end{split}
    \end{align}
    where the final equality above is given by~(\ref{eq:vardif}). The key step in the proof is now to bound the third derivative ${\dv{^3}{\theta^3}}(\hat{h}_G^u-\hat{n}_G^u)(\theta)$ using the inductive hypothesis. To begin, by the above expression for the third derivative,
    \begin{align}
      \label{eq:thirdderbound}
      \begin{split}
        \hspace{1em}&\hspace{-1em} \left|{\dv{^3}{\theta^3}}(\hat{h}_G^u-\hat{n}_G^u)(\theta)\right| \\
        &\leq \sum_{j\in\bZ}|j-p_1u|^3\cdot|(h_G^u)_j-(n_G^u)_j| \\
        &\leq (8\sqrt{u\log u})^3\cdot\|h_G^u-n_G^u\|_1 + \sum_{j\in\bZ:|j-p_1u|>8\sqrt{u\log u}}|j-p_1u|^3\cdot|(h_G^u)_j-(n_G^u)_j|.
      \end{split}
    \end{align}
    By the inductive hypothesis,
    \begin{align*}
      (8\sqrt{u\log u})^3\cdot\|h_G^u-n_G^u\|_1
      &\leq (8\sqrt{u\log u})^3 \cdot \frac{\lambda}{\sqrt{u}} \cdot (1+\log u)^{\eta_1 \log\log u+\eta_2} \\
      &= 2^9 \cdot \lambda u \cdot (1+\log u)^{\eta_1 \log\log u+\eta_2+3/2}.
    \end{align*}
    To bound the second term on the right hand side of~(\ref{eq:thirdderbound}), let $b^u\in[0,1]^{\bZ}$ denote the binomial distribution with parameters $u,p$. Then by Theorem~\ref{thm:difftail} and by condition~\ref{it:dntail} of Definition~\ref{def:discnorm} along with Corollary~\ref{cor:expasympvar},
    \begin{align*}
      \hspace{1em}&\hspace{-1em} \sum_{j\in\bZ:|j-p_1u|>8\sqrt{u\log u}}|j-p_1u|^3\cdot|(h_G^u)_j-(n_G^u)_j| \\
                  &\leq \sum_{a=8\sqrt{u\log u}}^\infty(a+1)^3\sum_{j\in\bZ:|j-p_1u|>a}(|(h_G^u)_j-b^u_j|+|b^u_j-(n_G^u)_j|) \\
                  &\leq \sum_{a=8\sqrt{u\log u}}^\infty(a+1)^3\left(4000 \cdot \lambda \cdot e^{-a^2/8u} + c_3 \cdot \frac{2\lambda}{1-\lambda} \cdot e^{-a^2/8u}\right) \\
                  &\leq \left(4000+\frac{2c_3}{1-\lambda}\right) \cdot \lambda \cdot \int_{a=8\sqrt{u\log u}-1}^\infty (2a)^3 e^{-a^2/8u} da \\
                  &\leq \left(4000+\frac{2c_3}{1-\lambda}\right) \cdot \lambda \cdot 2^8u^2\int_{q=6\log u}^\infty q e^{-q}dq \\
                  &= \left(4000+\frac{2c_3}{1-\lambda}\right) \cdot \lambda \cdot 2^8u^2 \cdot \frac{6\log u+1}{u^6} \\
                  &\leq \left(4000+\frac{2c_3}{1-\lambda}\right) \cdot \lambda \cdot 2^8
    \end{align*}
    where the third inequality above holds because $a^3e^{-a^2/8u}$ is maximized at $\sqrt{12 u}<8\sqrt{u\log u}-1$, the fourth inequality uses the fact that $8\sqrt{u\log u}-1\geq 7\sqrt{u\log u}$ and then applies the substitution $q=a^2/8u$, the equality follows by applying integration by parts, and the final inequality holds because $u\geq 2$ by assumption. Thus~(\ref{eq:thirdderbound}) becomes
    \begin{align*}
      \left|{\dv{^3}{\theta^3}}(\hat{h}_G^u-\hat{n}_G^u)(\theta)\right|
      &\leq 2^9 \cdot \lambda u \cdot (1+\log u)^{\eta_1 \log\log u+\eta_2+3/2} + 2^8\left(4000+\frac{2c_3}{1-\lambda}\right) \cdot \lambda \\
      &\leq C_2 \cdot \lambda u \cdot (1+\log u)^{\eta_1 \log\log u+\eta_2+3/2}.
    \end{align*}
    Combining the above inequality with~(\ref{eq:lowders}) and expanding the Taylor approximation at $\theta=0$ gives
    \begin{align*}
      \begin{split}
        |\hat{h}_G^u(\theta)-\hat{n}_G^u(\theta)|
        &\leq \frac{C_1}{2} \cdot \lambda \cdot |\theta|^2 + \frac{C_2}{6} \cdot \lambda u \cdot (1+\log u)^{\eta_1 \log\log u+\eta_2+3/2} \cdot |\theta|^3 \\
        &\leq \lambda \cdot \left(\frac{C_1}{2} \cdot |\theta|^2 + \frac{C_2}{6} \cdot u \cdot |\theta|^3\right) \cdot (1+\log u)^{\eta_1 \log\log u+\eta_2+3/2}.
      \end{split}
    \end{align*}
  \end{proof}

  We apply Lemma~\ref{lem:taylorbound} below to bound the Fourier coefficients of $\bigast_{k\in[\ell]}h_G^{t_k}-\bigast_{k\in[\ell]}n_G^{t_k}$.
  
  \begin{lemma}
    \label{lem:fourierbound}
    \begin{align*}
      \left|\prod_{k\in[\ell]}\hat{h}_G^{t_k}(\theta)-\prod_{k\in[\ell]}\hat{n}_G^{t_k}(\theta)\right|
      &\leq \frac{\lambda}{\sqrt{t}} \cdot (1+\log t)^{\eta_1 \log\log t+\eta_2} \cdot \frac{(1+\log t)^{6-\eta_1/20}}{2^{(\eta_2+6-\eta_1)/2}} \\
      &\hspace{.5cm}\cdot \left(C_1 \cdot t \cdot |\theta|^2 + \frac{C_2}{6} \cdot t^{3/2} \cdot |\theta|^3\right) \cdot e^{-p_0p_1t\theta^2/60}.
    \end{align*}
  \end{lemma}
  \begin{proof}
    By Lemma~\ref{lem:smooth},
    \begin{align*}
      |\hat{h}_G^u(\theta)| &\leq e^{-p_0p_1u\theta^2/20}.
    \end{align*}
    The same bound holds for $|\hat{n}_G^u(\theta)|$ by condition~\ref{it:dnfourier} of Definition~\ref{def:discnorm}. Thus
    \begin{align*}
      \hspace{1em}&\hspace{-1em}
                    \left|\prod_{k\in[\ell]}\hat{h}_G^{t_k}(\theta)-\prod_{k\in[\ell]}\hat{n}_G^{t_k}(\theta)\right|\\
                  &= \left|\sum_{k\in[\ell]}\left(\prod_{k'=0}^{k-1}\hat{n}_G^{t_{k'}}(\theta)\right)\cdot(\hat{h}_G^{t_k}(\theta)-\hat{n}_G^{t_k}(\theta))\cdot\left(\prod_{k'=k+1}^{\ell-1}\hat{h}_G^{t_{k'}}(\theta)\right)\right| \\
                  &\leq \sum_{k\in[\ell]}|\hat{h}_G^{t_k}(\theta)-\hat{n}_G^{t_k}(\theta)| \cdot e^{-p_0p_1(t-t_k)\theta^2/20} \\
                  &\leq \sum_{k\in[\ell]}|\hat{h}_G^{t_k}(\theta)-\hat{n}_G^{t_k}(\theta)| \cdot e^{-p_0p_1t\theta^2/60},
    \end{align*}
    where the final inequality above holds because $t-t_k\geq t/3$ by definition. Applying Lemma~\ref{lem:taylorbound} with $u=t_k$ for $k\in[\ell]$ to bound the sum above gives
    \begin{align*}
      \hspace{1em}&\hspace{-1em} \sum_{k\in[\ell]}|\hat{h}_G^{t_k}(\theta)-\hat{n}_G^{t_k}(\theta)| \\
                  &\leq \lambda\cdot\left(\frac{C_1}{2} \cdot \ell \cdot |\theta|^2 + \frac{C_2}{6} \cdot t \cdot |\theta|^3\right) \cdot \left(1+\log(\sqrt{t}+1)\right)^{\eta_1 \log\log\min\{\sqrt{t}+1,t-1\}+\eta_2+3/2} \\
                  &\leq \lambda \cdot \left(C_1 \cdot \sqrt{t} \cdot |\theta|^2 + \frac{C_2}{6} \cdot t \cdot |\theta|^3\right) \cdot \left(\frac{1+\log t}{\sqrt{2}}\right)^{\eta_1 (\log\log t-1/20)+\eta_2+6} \\
                  &\leq \lambda \cdot \left(C_1 \cdot \sqrt{t} \cdot |\theta|^2 + \frac{C_2}{6} \cdot t \cdot |\theta|^3\right) \cdot \frac{(1+\log t)^{6-\eta_1/20}}{2^{(\eta_2+6-\eta_1)/2}} \cdot (1+\log t)^{\eta_1 \log\log t+\eta_2}.
    \end{align*}
    where the first inequality above holds because all $k\in[\ell]$ have $t_k\leq\min\{\sqrt{t}+1,t-1\}$ by definition, and the second inequality above holds because $\ell\leq\sqrt{t}+1\leq 2\sqrt{t}$, and $t\geq 2$ so that $1+\log(\sqrt{t}+1)\leq((1+\log t)/\sqrt{2})^4$ and $\log\log\min\{\sqrt{t}+1,t-1\} \leq \log\log t-1/20$.

    The desired result now follows by combining the two inequalities above.
  \end{proof}

  The following lemma squares and integrates the Fourier coefficient bound in Lemma~\ref{lem:fourierbound} to bound the $\ell_2$-norm of $\bigast_{k\in[\ell]}h_G^{t_k}-\bigast_{k\in[\ell]}n_G^{t_k}$.
  
  \begin{lemma}
    \label{lem:convl2}
    \begin{align*}
      \left\|\bigast_{k\in[\ell]}h_G^{t_k}-\bigast_{k\in[\ell]}n_G^{t_k}\right\|
      &\leq t^{-1/4} \cdot \frac{\lambda}{\sqrt{t}} \cdot (1+\log t)^{\eta_1 \log\log t+\eta_2} \cdot \frac{(1+\log t)^{6-\eta_1/20}}{2^{(\eta_2+6-\eta_1)/2}} \cdot \frac{70 \cdot (C_1+C_2)}{(p_0p_1)^{7/2}}.
    \end{align*}
  \end{lemma}
  \begin{proof}
    Squaring and integrating the bound in Lemma~\ref{lem:fourierbound} gives
    \begin{align*}
      \left\|\bigast_{k\in[\ell]}h_G^{t_k}-\bigast_{k\in[\ell]}n_G^{t_k}\right\|
      &= \sqrt{\int_{\theta=-\pi}^\pi \left|\prod_{k\in[\ell]}\hat{h}_G^{t_k}(\theta)-\prod_{k\in[\ell]}\hat{n}_G^{t_k}(\theta)\right|^2 \frac{d\theta}{2\pi}} \\
      &\leq \frac{\lambda}{\sqrt{t}} \cdot (1+\log t)^{\eta_1 \log\log t+\eta_2} \cdot \frac{(1+\log t)^{6-\eta_1/20}}{2^{(\eta_2+6-\eta_1)/2}} \\
      &\hspace{.5cm}\cdot \sqrt{2C_1^2\int_{\theta=-\pi}^\pi t^2\theta^4 e^{-p_0p_1t\theta^2/30} \frac{d\theta}{2\pi} + \frac{C_2^2}{18}\int_{\theta=-\pi}^\pi t^3\theta^6 e^{-p_0p_1t\theta^2/30} \frac{d\theta}{2\pi}}.
    \end{align*}
    Substituting $q=\sqrt{p_0p_1t/30}\cdot\theta$ in the integrals in the right  hand side above gives
    \begin{align*}
      \hspace{1em}&\hspace{-1em} \sqrt{2C_1^2\int_{\theta=-\pi}^\pi t^2\theta^4 e^{-p_0p_1t\theta^2/30} \frac{d\theta}{2\pi} + \frac{C_2^2}{18}\int_{\theta=-\pi}^\pi t^3\theta^6 e^{-p_0p_1t\theta^2/30} \frac{d\theta}{2\pi}} \\
                  &\leq \sqrt{2C_1^2\left(\frac{30}{p_0p_1}\right)^{5/2} \cdot t^{-1/2} \cdot \int_{q=-\infty}^\infty q^4e^{-q^2} \frac{dq}{2\pi} + \frac{C_2^2}{18}\left(\frac{30}{p_0p_1}\right)^{7/2} \cdot t^{-1/2} \cdot \int_{q=-\infty}^\infty q^6e^{-q^2} \frac{dq}{2\pi}} \\
                  &\leq t^{-1/4} \cdot \frac{70 \cdot (C_1+C_2)}{(p_0p_1)^{7/2}}.
    \end{align*}
    The two inequalities above imply the desired bound.
  \end{proof}

  To complete the proof, we apply the Cauchy-Schwartz inequality to the bound in Lemma~\ref{lem:convl2} to obtain an $\ell_1$-bound, which combines with Lemma~\ref{lem:startind} to imply the theorem statement. Specifically, by Lemma~\ref{lem:convl2} and Cauchy-Schwartz,
  \begin{align*}
    &\left\|\left(\bigast_{k\in[\ell]}h_G^{t_k}-\bigast_{k\in[\ell]}n_G^{t_k}\right)_S\right\|_1 \\
    &\leq \sqrt{16\sqrt{t\log t}+1} \cdot \left\|\bigast_{k\in[\ell]}h_G^{t_k}-\bigast_{k\in[\ell]}n_G^{t_k}\right\| \\
    &\leq \frac{\lambda}{\sqrt{t}} \cdot (1+\log t)^{\eta_1 \log\log t+\eta_2} \cdot \frac{(1+\log t)^{7-\eta_1/20}}{2^{(\eta_2+6-\eta_1)/2}} \cdot \frac{300 \cdot (C_1+C_2)}{(p_0p_1)^{7/2}}.
  \end{align*}
  Applying this inequality with Lemma~\ref{lem:startind} gives
  \begin{align*}
    \hspace{1em}&\hspace{-1em} \|h_G^t-n_G^t\|_1 \\
                &\leq \frac{\lambda}{\sqrt{t}} \cdot \left(8000 + \frac{2(c_2+c_3)}{1-\lambda} + (1+\log t)^{\eta_1 \log\log t+\eta_2} \cdot \frac{(1+\log t)^{7-\eta_1/20}}{2^{(\eta_2+6-\eta_1)/2}} \cdot \frac{300 \cdot (C_1+C_2)}{(p_0p_1)^{7/2}}\right) \\
    &\leq \frac{\lambda}{\sqrt{t}}(1+\log t)^{\eta_1 \log\log t+\eta_2} \cdot \left(\frac{8000 + 2(c_2+c_3)/(1-\lambda)}{(1+\log 2)^{\eta_1\log\log t+\eta_2}} + \frac{(1+\log t)^{7-\eta_1/20}}{2^{(\eta_2+6-\eta_1)/2}} \cdot \frac{300 \cdot (C_1+C_2)}{(p_0p_1)^{7/2}}\right).
  \end{align*}
  Thus $\|h_G^t-n_G^t\|_1 \leq (\lambda/\sqrt{t})(1+\log t)^{\eta_1 \log\log t+\eta_2}$ when
  \begin{align*}
    \eta_1 &= 140 \\
    \eta_2 &= 140 + \frac{2}{\log 2} \cdot \log\max\left\{8000+3(c_2+c_3),\; \frac{300(C_1+C_2)}{(p_0p_1)^{7/2}}\right\} \\
           &\leq 140 + 3 \cdot \log\left(\frac{2^{28}+2^{10}c_1+2^{18}c_3}{(p_0p_1)^{7/2}} + 3c_2\right),
  \end{align*}
  where we have applied the fact that $C_1+C_2 \leq 2^{20} + 3c_1 + 2^{10}c_3$ by definition. Thus we have shown that~(\ref{eq:beinduct}) holds, completing the inductive step of the proof.
\end{proof}

\section{Acknowledgments}
The author thanks Salil Vadhan for numerous helpful discussions and suggestions, and thanks Venkat Guruswami for helping to motivate this work and providing feedback.

This work was primarily done while the author was at Harvard University. The author is currently supported by a National Science Foundation Graduate Research Fellowship under Grant No.~DGE 2146752.

\bibliography{library}
\bibliographystyle{alpha}

\appendix
\renewcommand{\thesection}{\Alph{section}}

\section{Omitted proofs}
\label{app:beproofs}
This section presents the proofs we omitted in Section~\ref{sec:prelim}. These proofs are standard or follow directly from prior work, but we include them here for completeness.

\begin{proof}[Proof of Lemma~\ref{lem:smooth}]
  This result was implicitly shown by Golowich and Vadhan~\cite{golowich_pseudorandomness_2022-1} in the proof of Theorem~\ref{thm:difftail}. Specifically, define $F\in\bR^{V\times V}$ by $F=J+(I-J)/10$, and define $P_\theta^{(0)}\in\bC^{V\times V}$ to be the diagonal matrix with $(P_\theta^{(0)})_{v,v}=e^{-i\theta(\val(v)-p_1)}$. Then the proof of Lemma~26 in \cite{golowich_pseudorandomness_2022} (the full version of \cite{golowich_pseudorandomness_2022-1}) shows that
  \begin{align*}
    |\bE[e^{-i\theta\Sigma\val(\RW_{G}^t)}]|
    &= |\vec{1}^\top(GP_\theta^{(0)})^t\vec{1}|.
  \end{align*}
  Now
  \begin{align*}
    |\vec{1}^\top(GP_\theta^{(0)})^t\vec{1}|
    &= |\vec{1}^\top(F^{-1}GF^{-1}\cdot FP_\theta^{(0)}F)^t\vec{1}| \\
    &\leq (\|F^{-1}GF^{-1}\|\cdot\|FP_\theta^{(0)}F\|)^t.
  \end{align*}
  Because $\lambda(G)\leq 1/100$ and $F^{-1}=J+10(I-J)$, it follows that $\|F^{-1}GF^{-1}\|\leq 1$, so
  \begin{align*}
    |\bE[e^{-i\theta\Sigma\val(\RW_{G}^t)}]|
    &\leq \|FP_\theta^{(0)}F\|^t.
  \end{align*}
  Lemma~28 in \cite{golowich_pseudorandomness_2022} implies that $\|FP_\theta^{(0)}F\|\leq e^{-p_0p_1\theta^2/20}$, so we obtain the desired inequality
  \begin{align*}
    |\bE[e^{-i\theta\Sigma\val(\RW_{G}^t)}]|
    &\leq e^{-p_0p_1t\theta^2/20}.
  \end{align*}
\end{proof}

\begin{proof}[Proof of Corollary~\ref{cor:expasympvar}]
  By Lemma~\ref{lem:asympvardef},
  \begin{align*}
    |\sigma^2(\Sigma\val(\RW_G^t)) - p_0p_1|
    &\leq 2\sum_{i=1}^\infty \frac{1}{n}|(\val-p_1)^\top G^i(\val-p_1)| \\
    &\leq 2\sum_{i=1}^\infty \lambda^i \cdot \frac{\|\val-p_1\|^2}{n} \\
    &= \frac{2}{1-\lambda} \cdot \lambda \cdot p_0p_1,
  \end{align*}
  where the first inequality above holds because $(\val-p_1)^\top G^i\vec{1}=(\val-p_1)^\top\vec{1}=0$ so that $(\val-p_1)^\top G^i\val=(\val-p_1)^\top G^i(\val-p_1)$, and the second inequality holds becuase $\lambda(G)=\lambda$ and $\val-p_1\in\vec{1}^\perp$.
\end{proof}

In the (standard) proof below of Lemma~\ref{lem:asympvarconv}, by bounding the rate of convergence to the expression in Lemma~\ref{lem:asympvardef}, we also implicitly prove Lemma~\ref{lem:asympvardef}.

\begin{proof}[Proof of Lemma~\ref{lem:asympvarconv}]
  By definition
  \begin{align*}
    \Var(\Sigma\val(\RW_G^t))
    &= \bE[(\Sigma\val(\RW_G^t))^2] - \bE[\Sigma\val(\RW_G^t)]^2 \\
    &= \sum_{i\in[t]}\sum_{i'\in[t]}\bE[\val(\RW_G^t)_i\cdot\val(\RW_G^t)_{i'}] - (p_1t)^2.
  \end{align*}
  The sum on the right hand side above can simplified as follows. For all $i=i'\in[t]$, then $\bE[\val(\RW_G^t)_i\cdot\val(\RW_G^t)_{i'}]=p_1$. Otherwise, if $i<i'$, then $\bE[\val(\RW_G^t)_i\cdot\val(\RW_G^t)_{i'}]=\bE[\val(\RW_G^t)_0\cdot\val(\RW_G^t)_{i'-i}]$, with an analogous equality if $i>i'$. Thus
  \begin{align*}
    \Var(\Sigma\val(\RW_G^t))
    &= p_1t + 2\sum_{\ell=1}^{t-1}(t-\ell)\bE[\val(\RW_G^t)_0\cdot\val(\RW_G^t)_\ell] - (p_1t)^2 \\
    &= (p_1-p_1^2)t + 2\sum_{\ell=1}^{t-1}(t-\ell)\bE[\val(\RW_G^t)_0\cdot\val(\RW_G^t)_\ell-p_1^2] \\
    &= p_0p_1t + 2\sum_{\ell=1}^{t-1}(t-\ell)\bE[\val(\RW_G^t)_0\cdot(\val(\RW_G^t)_\ell-p_1)].
  \end{align*}
  For $v\in V$, conditioned on the event that $(\RW_G^t)_0=v$, then the distribution of $(\RW_G^t)_\ell$ is given by $G^\ell\1_v$. Thus
  \begin{equation*}
    \bE[\val(\RW_G^t)_\ell-p_1\mid(\RW_G^t)_0=v] = (\val-p_1)^\top G^\ell\1_v,
  \end{equation*}
  so
  \begin{align*}
    \Var(\Sigma\val(\RW_G^t))
    &= p_0p_1t + 2\sum_{\ell=1}^{t-1}(t-\ell)\sum_{v\in V}\frac{1}{n}\val(v)\cdot (\val-p_1)^\top G^\ell\1_v \\
    &= p_0p_1t + 2\sum_{\ell=1}^{t-1}(t-\ell)\frac{1}{n}(\val-p_1)^\top G^\ell\val.
  \end{align*}

  The difference between $1/t$ times the above expression and the asymptotic variance as given in Lemma~\ref{lem:asympvardef} is
  \begin{align*}
    \hspace{-1em}&\hspace{-1em}\sigma^2(\Sigma\val(\RW_G^t)) - \frac{1}{t}\Var(\Sigma\val(\RW_G^t)) \\
                 &= 2\sum_{i=1}^{t-1}\frac{i}{t}\cdot\frac{1}{n}(\val-p_1)^\top G^i\val + 2\sum_{i=t}^\infty\frac{1}{n}(\val-p_1)^\top G^i\val.
  \end{align*}
  Because $G$ is a $\lambda$-spectral expander and $\val-p_1\in\bR^V$ is orthogonal to $\vec{1}$,
  \begin{align*}
    \left|\frac{1}{n}(\val-p_1)^\top G^i\val\right|
    &= \left|\frac{1}{n}(\val-p_1)^\top G^i(\val-p_1)\right| \\
    &= \left|\frac{1}{\sqrt{n}}(\val-p_1)^\top G^i\frac{1}{\sqrt{n}}(\val-p_1)\right| \\
    &\leq \left\|\frac{1}{\sqrt{n}}(\val-p_1)^\top\right\|\lambda^i\left\|\frac{1}{\sqrt{n}}(\val-p_1)\right\| \\
    &= \lambda^i \cdot p_0p_1.
  \end{align*}
  Thus
  \begin{align*}
    \left|\sigma^2(\Sigma\val(\RW_G^t)) - \frac{1}{t}\Var(\Sigma\val(\RW_G^t))\right|
    &\leq 2\sum_{i=1}^{t-1}\frac{i}{t}\cdot\lambda^i \cdot p_0p_1 + 2\sum_{i=t}^\infty\lambda^i \cdot p_0p_1 \\
    &\leq 2\sum_{i=1}^\infty\frac{i}{t}\cdot\lambda^i \cdot p_0p_1 \\
    &= \frac{2}{(1-\lambda)^2}\cdot\frac{\lambda}{t} \cdot p_0p_1.
  \end{align*}
\end{proof}

\begin{proof}[Proof of Proposition~\ref{prop:stickydn}]
  By Lemma~\ref{lem:stickyvar}, $\sigma^2=\sigma^2(\Sigma\val(\RW_{G_{\lambda,p}}^t))$, that is, the choice of $\lambda$ is such that $\Sigma\val(\RW_{G_{\lambda,p}}^t)$ has asymptotic variance $\sigma^2=p_0p_1(1+\lambda)/(1-\lambda)$. Also by definition, $\sigma^2-p_0p_1 = p_0p_1 \cdot 2\lambda/(1-\lambda)$, and $-\min\{p_0/p_1,p_1/p_0,1/100\}\leq\lambda\leq 1/100$. We apply these facts below to show that $\cN_{\sigma^2}^t = \Sigma\val(\RW_{G_{\lambda,p}}^t)$ satisfies each of the conditions in Definition~\ref{def:discnorm}:
  \begin{enumerate}
  \item By definition $\Sigma\val(\RW_{G_{\lambda,p}}^t)=p_1t$.
  \item By Lemma~\ref{lem:asympvarconv},
    \begin{align*}
      |\Var(\Sigma\val(\RW_{G_{\lambda,p}}^t))-\sigma^2t|
      &\leq \frac{2}{(1-|\lambda|)^2} \cdot |\lambda| \cdot p_0p_1 \\
      &= \frac{2}{(1-|\lambda|)^2} \cdot \frac{|\sigma^2-p_0p_1|\cdot(1-\lambda)}{2} \\
      &\leq \frac{101/100}{(99/100)^2} \cdot |\sigma^2-p_0p_1|.
    \end{align*}
  \item By definition $\Sigma\val(\RW_{G_{\lambda,p}}^1) = \text{Bern}(p)$.
  \item By definition, the distribution of $\sum_{i\in[\ell]}\Sigma\val(\RW_{G{\lambda,p}}^{t_k})$ is equal to the distribution of $\Sigma\val(\RW_G^t)$ for a sequence $G=(G_i)_{1\leq i\leq t-1}$ of expander graphs, $\ell-1$ of which are $J$ and $t-\ell$ of which are $G_{\lambda,p}$. Thus because $\|G_{\lambda,p}-J\|=\lambda(G_{\lambda,p})=|\lambda|$, Theorem~\ref{thm:difftail} implies that
    \begin{align*}
      2 \cdot \disTV\left(\sum_{i\in[\ell]}\Sigma\val(\RW_{G{\lambda,p}}^{t_k}),\; \Sigma\val(\RW_{G_{\lambda,p}}^t)\right)
      &\leq 4000 \cdot |\lambda| \cdot \frac{\ell-1}{t} \\
      &\leq 2000(1-\lambda) \cdot |\sigma^2/p_0p_1-1| \cdot \frac{\ell-1}{t} \\
      &\leq 2020 \cdot |\sigma^2/p_0p_1-1| \cdot \frac{\ell-1}{t}.
    \end{align*}
  \item By Theorem~\ref{thm:difftail},
    \begin{align*}
      \hspace{1em}&\hspace{-1em} \sum_{j\in\bZ:|j-p_1t|\geq a}|\Pr[\Sigma\val(\RW_{G_{\lambda,p}}^t)=(t-j,j)]-\Pr[\text{Bin}(t,p)=j]| \\
      &\leq 4000 \cdot |\lambda| \cdot e^{-a^2/8t} \\
      &\leq 2020 \cdot |\sigma^2/p_0p_1-1| \cdot e^{-a^2/8t}.
    \end{align*}
  \item This fact follows directly from Lemma~\ref{lem:smooth}.
  \end{enumerate}
\end{proof}

\end{document}